\documentclass{amsart}
\usepackage{color}
\usepackage{graphicx}
\usepackage{frcursive}
\newcommand{\dt}{\partial_t}
\newcommand{\dx}{\partial_x}

\newcommand{\eps}{\varepsilon}
\newcommand{\R}{{\mathbb R}}

\newtheorem{definition}{Definition}
\newtheorem{theorem}{Theorem}
\newtheorem{proposition}{Proposition}
\newtheorem{lemma}{Lemma}
\theoremstyle{remark}
\newtheorem{remark}{Remark}

\title{Generating boundary conditions for a Boussinesq system}
\author{D. Lannes}
\address{Institut de Math\'ematiques de Bordeaux et CNRS, UMR 5251}
\email{David.Lannes@math.u-bordeaux.fr}
\author{L. Weynans}
\address{Institut de Math\'ematiques de Bordeaux et CNRS, UMR 5251}
\email{Lisl.Weynans@math.u-bordeaux.fr}
\thanks{D. L. is supported by the Fondation Del Duca de l'Acad\'emie des Sciences, the ANR grants ANR-17-CE40-0025 NABUCO and ANR-18-CE40-0027-01 Singflows, and the Conseil R\'egional d'Aquitaine.}
\begin{document}
\maketitle

\begin{abstract}
We present a new method for the numerical implementation of generating boundary conditions for a one dimensional Boussinesq system. This method is based on a reformulation of the equations and a resolution of the dispersive boundary layer that is created at the boundary when the boundary conditions are non homogeneous. This method is implemented for a simple first order finite volume scheme and validated by several numerical simulations. Contrary to the other techniques that can be found in the literature, our approach does not cause any increase in computational time with respect to periodic boundary conditions.
\end{abstract}
\section{Introduction}

\subsection{General setting}

Among the many reduced models used to describe the evolution of waves at the surface of a fluid in shallow water, the nonlinear shallow water equations are certainly one of the most used for applications. They can be written in conservative form as
\begin{equation}\label{WE1ter}
\begin{cases}
\dt \zeta +\dx q=0,\\
\dt q+\dx \big(\frac{1}{2}gh^2+\frac{1}{h}q^2\big)=0,
\end{cases}
\end{equation}
where $\zeta$ is the surface elevation above the rest state and $q$ the horizontal discharge (equivalently, the vertical integral of the horizontal velocity), and where $h=H_0+\zeta$ is the total water depth ($H_0$ being the depth at rest) and $g$ the acceleration of gravity.\\
For many applications, the surface elevation is known at the entrance of the domain (through buoy measurements for instance, or it can be provided by offshore swell models) and is imposed as a boundary condition for the model 
\begin{equation}\label{CB}
\zeta(t,x=0)=f(t), \quad \mbox{ for all } t\geq 0,
\end{equation}
as well as the initial values for $q$ and $\zeta$ in the domain,
\begin{equation}\label{CI}
(\zeta,q)(t=0,x)=(\zeta^0,q^0)(x),\quad \mbox{ for all } x\geq 0;
\end{equation}
this type of boundary condition is often referred to as {\it generating} boundary condition (see for instance \cite{DS}). It is used a lot in coastal oceanography, where the offshore swell is imposed at the entrance of the domain of interest (see for instance the classical benchmark \cite{Dingemans}).\\
The problem consisting in solving \eqref{WE1ter} together with \eqref{CB} and \eqref{CI} is a {\it mixed initial-boundary value problem (IBVP)}; due to its hyperbolic nature,  it can be solved theoretically (see for instance \cite{LY,PT}, and more recently \cite{IguchiLannes} for sharp well-posedness results). From the numerical viewpoint, solving this IBVP is also possible, using the decomposition of the solution into Riemann invariants (see for instance \cite{Marche}).

The nonlinear shallow water equations provide a robust model used in many applications; it is known \cite{AL,Iguchi,Lannes_book,Lannes_SW} to provide an approximation of the full free surface Euler equations with a precision $O(\mu)$, where $\mu=H_0^2/L^2$ is the shallowness parameter ($L$ denotes here the typical horizontal scale of the waves). It omits however the dispersive effects that play an important role in coastal areas, in particular in the shoaling zone; in order to take them into account, one has to keep the $O(\mu^2)$ terms that are neglected in the derivation of the nonlinear shallow water equations. The most simple models that reach such a precision and therefore take into account the dispersive effects while retaining nonlinear terms are the so-called Boussinesq models. There are actually many asymptotically equivalent Boussinesq models \cite{BCS1,BCS2,BCL}; their simplicity is due to the fact that they are derived under the assumption that   the waves are  of small amplitude compared to the depth, which allows to neglect some of the nonlinear terms (without this assumption, one has to work with the much more complicated Serre-Green-Naghdi equations, see \cite{Lannes_book} for instance). We choose here to work with the so-called Boussinesq-Abbott model \cite{Abbott,FBCR} since its structure is obviously a dispersive perturbation of the nonlinear shallow water equations,
\begin{equation}\label{Bouss1ND_int}
\begin{cases}
\dt \zeta +\dx q =0,\\
(1-\frac{H_0^2}{3}\dx^2)\dt q+\dx \big(\frac{1}{2} gh^2 + \frac{1}{h}q^2\big)=0,\qquad (h=H_0+\zeta)
\end{cases}
\end{equation}
(removing the dispersive term $ -\frac{H_0^2}{3}\dx^2\dt q$, this model reduces to \eqref{WE1ter}). As above, we are interested in the initial-boundary value problem for this system, we therefore complement it with the boundary condition \eqref{CB} and initial condition \eqref{CI}. Contrary to \eqref{WE1ter}, this system is no longer hyperbolic, and there is no general theory to address the IBVP. Only some particular cases have been considered, such as in \cite{Xue} with homogeneous boundary conditions, \cite{BC,ADM} for a particular class of Boussinesq systems (the Bona-Smith family) where a regularizing dispersion is also present in the first equation, \cite{LCZ} for a higher order Boussinesq system or \cite{LM} for the shoreline problem (vanishing depth). \\
Due to its importance for numerical simulations in coastal oceanography, there has been a significant amount of work devoted to finding numerical answers to this issue in recent years. For the related problem of transparent boundary conditions in particular (i.e. which boundary conditions should be put at the boundary of the computational domain so that waves can pass through this artificial boundary without being affected by it), the linear problem has been considered for scalar equations (such as KdV or BBM) in \cite{BMN,BNS} as well as for the linearization of \eqref{Bouss1ND_int} around the rest state. For the nonlinear case, a different approach has been used recently, which consists in implementing a perfectly matched layer (PML) approach for a hyperbolic relaxation of the Green-Naghdi equations \cite{Kazakova}. This approach can be used to deal with generating boundary conditions such as \eqref{CB} but the size of the layer in which the PML approach is implemented is typically of two wavelength, which for applications to coastal oceanography can typically represent an increase of $100\%$ of the computational domain. Other methods such as the source function method \cite{WKS} also require a significant increase of computational time.

Our goal in this note is to propose a new approach to the nonlinear Boussinesq system \eqref{Bouss1ND_int} with generating boundary condition \eqref{CB} and initial condition \eqref{CI}, and which does not require any extension of the computational domain. It is based in a reformulation of the problem \eqref{Bouss1ND_int} and \eqref{CB}-\eqref{CI} into a non-homogeneous system of conservation laws for $\zeta$ and $q$ with a nonlocal flux, and with  a source term accounting for the dispersive boundary layer,
\begin{equation}\label{WE1terND_reform_intro}
\begin{cases}
\displaystyle \dt \zeta +\dx q =0,\\
\displaystyle \dt q+\dx R_1  \big( \frac{1}{2} g(h^2-H_0^2)+\frac{1}{h}q^2 \big)= \underline{{\mathcal Q}}(\underline{q},f,\ddot f, \zeta, q)\exp\big(-\sqrt{3}\frac{x}{H_0}\big),
\end{cases}
\end{equation}
where $\underline{q}=q_{\vert_{x=0}}$ and
\begin{align}
\nonumber
\underline{{\mathcal Q}}(\underline{q},f,\ddot f, \zeta, q)=&\frac{\sqrt{3}}{H_0}\frac{\underline{q}^2}{H_0+ f}+\frac{H_0}{\sqrt{3}} \ddot f+\frac{g\sqrt{3}}{H_0}(H_0+\frac{1}{2} f)f \\
 \label{defODEint}
& - \frac{\sqrt{3}}{H_0} \underline{R}_1 \big( \frac{1}{2}g (h^2-H_0^2)+\frac{1}{h}q^2 \big),
\end{align}
with the initial condition
\begin{equation}\label{IBC_Bouss_bis_int}
(\zeta,q)(t=0,x)=(\zeta^0,q^0)(x),\qquad \underline{q}(t=0)=q^0(x=0),
\end{equation}
and the boundary condition
\begin{equation}\label{Bouss_reformBC_int}
\zeta(t,x=0)=f(t);
\end{equation}
here, we denoted by $R_1$ the inverse of the operator $(1-\frac{H_0^2}{3} \dx^2)$ on $\R_+$ with homogeneous Neumann boundary condition, and $\underline{R}_1 f =(R_1 f)_{\vert_{x=0}}$ (see Definition \ref{defR} below). The source term in the equation for $q$ is a dispersive boundary layer that appears because the time derivative of the trace $\underline{q}$ of $q$ at $x=0$ does not necessarily vanish. 

We propose here a simple numerical scheme based on this new formulation of the problem (which we prove to be well-posed), very easy to implement and that does not require to work on an extended computational domain. The ability of this method to handle generating boundary conditions with a very good precision is illustrated by several computations where nonlinear and dispersive terms both play an important role.

\subsection{Organization of the paper}

We describe in this paper how to handle a generating boundary condition on the left border of the computational domain. 
For the sake of clarity, we consider the problem on the half line $(0,\infty)$ so that we do not have to deal with boundary conditions at the right boundary; for our numerical simulations, we either consider a domain that is large enough so that the influence of the right boundary condition  is negligible, or take a wall boundary condition $q=0$ on the right boundary (of course, a generating boundary condition on the right-boundary can be handled by a straightforward adaptation of what is done at the left boundary).

\medbreak

In Section \ref{secthyp} we briefly recall the theory and numerical simulation of generating boundary conditions for the nonlinear shallow water equations in order to make clear that different mechanisms are at stake in the hyperbolic (shallow water) and dispersive (Boussinesq) cases. Note in particular that for the hyperbolic case considered in this section, the missing data at the boundary (i.e. the trace of the discharge $q$ at $x=0$) is deduced from the value at the boundary of the outgoing Riemann invariant which can itself be determined in terms of interior values by solving the characteristic equation.

In Section \ref{sectBou} we consider generating boundary conditions for the Boussinesq system \eqref{Bouss1ND_int}. In order to make more apparent the structure of the dispersive boundary layer we shall construct, we first non-dimensionalize the equations in \S \ref{sectdim}. The main step of the analysis is performed in \S \ref{sectreform} where the dispersive boundary layer is constructed and the reformulation \eqref{WE1terND_reform_intro} of the problem is derived. This reformulation is used in \S \ref{sectWP} to prove the local well-posedness of the initial-boundary value problem for the Boussinesq system \eqref{Bouss1ND_int} with generating boundary condition \eqref{CB} which, to our knowledge, was not known so far. A discretization of the reformulation \eqref{WE1terND_reform_intro} is then proposed in \S \ref{sectdisc}; for the sake of clarity, it is based on the standard Lax-Friedrichs scheme. It is not possible to recover the missing boundary data using Riemann invariants as in the hyperbolic case, but the nonlocality of the operator $R_1$ allows us to express this missing information in terms of interior values. We insist on the fact that our numerical treatment of the generating boundary condition does not increase the computational time compared to simple boundary conditions (periodic, physical well, etc.), contrary to the previously used approaches mentioned in the introduction.

Finally, we provide in Section \ref{sectnumval} several numerical computations showing the accuracy of our numerical scheme. Our validation method is first presented in \S \ref{sectvalmet}; it consists in computing a reference solution in a large domain $[-L,L]$ with a very refined mesh, and to use the values of the water elevation at $x=0$ provided by this solution as a generating boundary condition for computations on the small domain $[0,L]$. The accuracy of this new solution  is measured by comparing it with the reference solution. A first example is provided in \S \ref{refined} in a situation where both incoming and outgoing waves are present. In \S \ref{soliton} we show that the Boussinesq system \eqref{Bouss1ND_int} admits solitary waves, and that we are able to generate them with good accuracy using the corresponding generating boundary condition. We then provide in \S \ref{sinus} another example, relevant for applications to coastal oceanography \cite{BMF,RF,LannesMarche}, which consists in the generation of  a sinusoidal wave train.

\subsection{Notations}
For the numerical computations, the computational domain $[0,L]$ is discretized using a uniform grid: 
$$ x_0 = 0, x_1 = \delta_x, \dots , x_{i} =i \delta_x, \dots, x_{n_x-1} = (n_x-1) \delta_x, x_{n_x} = L,$$
with $\delta_x = \displaystyle{\frac{L}{n_x}}$. 
The time step is denoted $\delta_t$.
The variables $\zeta_i^n$ and $q^n_i$ denote the values of the numerical solution for $\zeta$ and $q$ at the time $ n \, \delta_t$ and at the location $x_i$.
Generally speaking, the subscript $i$ and the  superscript $n$ indicate respectively a discretization at the location  $x_i$ and at the time  $ n \, \delta_t$.

%


\section{The nonlinear shallow water equations}\label{secthyp}
We recall that the nonlinear shallow water (or Saint-Venant) equations are a system of equations coupling the surface elevation $\zeta$ above the rest state to the horizontal discharge $q$,
\begin{equation}\label{WE1ter_bis}
\begin{cases}
\dt \zeta +\dx q=0,\\
\dt q+\dx \big(\frac{1}{2}gh^2+\frac{1}{h}q^2\big)=0,
\end{cases}
\qquad (h=H_0+\zeta).
\end{equation}
This system of equations is complemented by the initial and boundary conditions
\begin{equation}\label{IBC}
(\zeta,q)(t=0,x)=(\zeta^0,q^0)(x),\qquad
\zeta(t,x=0)=f(t).
\end{equation}

For the sake of completeness and as a basis for comparisons with the dispersive (Boussinesq) case, we briefly recall here how this problem can be handled numerically.
\subsection{The Riemann invariants}
The nonlinear shallow water equations \eqref{WE1ter_bis} can be written under an equivalent quasilinear form by introducing the vertically averaged horizontal velocity $u$,
$$
u=\frac{q}{h}\qquad \mbox{ with }\quad h=H_0+\zeta.
$$
The resulting system of equations on $(\zeta,u)$ is given by
\begin{equation}\label{WE1}
\begin{cases}
\dt \zeta +\dx (hu)=0,\\
\dt u+g\dx \zeta +u\dx u=0
\end{cases}
\end{equation}
or, in more condensed form,
\begin{equation}
\label{WE1bis}
\dt U+A(U)\dx U=0\quad \mbox{ with }\quad U=(\zeta,u)^T\quad\mbox{ and }\quad A(U)=\left(\begin{array}{cc} u & h \\g & u\end{array}\right). 
\end{equation}
The matrix $A(U)$ is diagonalizable with eigenvalues $\lambda_+(U)$ and $-\lambda_-(U)$  and associated left-eigenvectors ${\bf e}_\pm(U)$ given by
$$
\lambda_\pm(U)=\pm u + \sqrt{gh} \quad \mbox{ and }\quad {\bf e}_\pm(U)=\big(\sqrt{\frac{g}{h}},\pm1\big)^T.
$$
Taking the scalar product of \eqref{WE1bis} and ${\bf e}_\pm(U)$, we obtain 
$$
\big(\sqrt{\frac{g}{h}}\dt h\pm\dt u\big) \pm (\pm u + \sqrt{gh})\big(\sqrt{\frac{g}{h}}\dx h\pm\dx u\big)=0.
$$
This leads us to introduce the Riemann invariants $R_\pm$ as
\begin{equation}\label{Riemann_inv}
R_\pm(U):= 2\big(\sqrt{gh}-\sqrt{gH_0})\pm u,
\end{equation}
which satisfy the transport equations
\begin{equation}\label{transport}
\dt R_++\lambda_+(U)\dx R_+=0,\qquad \dt R_- -\lambda_-(U)\dx R_-=0.
\end{equation}
These Riemann invariants play a central role in the numerical resolution of the IBVP \eqref{WE1ter_bis}-\eqref{IBC} presented below.
\subsection{The discrete equations}
Writing $U=(\zeta,q)^T$, we first write \eqref{WE1ter_bis}  in the condensed form,
\begin{equation}\label{WE2}
\dt U+\dx \big(F(U) \big)=0  \quad \mbox{ with } \quad  F=\big(q,\frac{1}{2}g(h^2-H_0^2)+\frac{1}{h}q^2\big)^T,
\end{equation}
for which a finite volume type discretization gives
\begin{equation}\label{FV}
\frac{U_i^{n+1}-U_i^n}{\delta_t}+\frac{1}{\delta_x}(F^n_{i+1/2}-F^n_{i-1/2})=0,\qquad i\geq 1,
\end{equation}
the choice of $F^n_{i+1/2}$ depending on the numerical scheme. Our focus here being on explaining how to handle the boundary condition \eqref{IBC}, we consider here the most simple case of the Lax-Friedrichs scheme where the discrete flux is given by
\begin{equation}\label{LF}
F^n_{i-1/2}=\frac{1}{2}(F^n_{i}+F^n_{i-1})-\frac{\delta_x}{2\delta_t}(U_i^n-U_{i-1}^n),\qquad i\geq 1,
\end{equation}
with $F^n_i=F(U^n_i)$. For $i=1$, this equation involves $U^n_0$  that we need to express in terms of $U^n=(U^n_i)_{1\leq i }$ and the initial-boundary condition \eqref{IBC}, which, in discrete form, reads
\begin{equation}\label{IBCdisc}
(\zeta^0_i,q^0_i)=(\zeta^0,q^0)(x_i) \quad (i\geq 1),\qquad \zeta^n_0=f^n:=f(t^n),
\end{equation}
with $t^n:= n\delta_t$; this is done in the following section.
\subsection{Data on the water depth on the left boundary}\label{sectDATA}
For $i=1$, the flux $F_{1/2}$ requires the knowledge of $U^n_0=(\zeta^n_0,q^n_0)$.
From the initial-boundary condition \eqref{IBCdisc}, one takes
$$
\zeta^n_0=f^n,
$$
but we need to determine $q_0^n$, which can be deduced from the knowledge of $R^n_{\pm,0}:=R_\pm(t^n,0)$.  From \eqref{Riemann_inv} one gets indeed
$$
q=\frac{h}{2}(R_+-R_-) \quad \mbox{ and }\quad R_++R_-=4\big(\sqrt{gh}-\sqrt{gH_0})
$$
and therefore
$$
q=h\big(2(\sqrt{gh}-\sqrt{gH_0})-R_-\big).
$$
Evaluating this relation at $x=0$ provides an expression for the trace $\underline{q}=q_{\vert_{x=0}}$ in terms of the boundary data $f=\zeta_{\vert_{x=0}}$ and of the trace of the outgoing Riemann invariant $R_-$,
\begin{equation}\label{qhyp}
\underline{q}=(H_0+f)\big(2\big(\sqrt{g(H_0+f)}-\sqrt{gH_0}\big)-{R_-}_{\vert_{x=0}}\big)
\end{equation}
and at the discrete level, we get  at $t=t^n$
\begin{equation}\label{qn0}
q^n_0=(H_0+f^n)\big(2(\sqrt{g(H_0+f^n)}-\sqrt{gH_0})-R_{-,0}^n\big).
\end{equation}
Therefore, we just need to determine $R_{-,0}^n$ in order to determine $q_0^n$. 
We use the characteristic equation \eqref{transport} satisfied by $R_-$; after discretization, this gives
$$
\frac{R_{-,0}^{n}-R^{n-1}_{-,0}}{\delta_t}- \lambda^{n-1}_{-} \frac{R_{-,1}^{n-1}-R_{-,0}^{n-1}}{\delta_x}=0; 
$$
as in \cite{Marche},  $\lambda^{n-1}_{-}$ is computed as a linear interpolation between 
$ \lambda_{-}(U^{n-1}_0)$ 
and 
$ \lambda_{-}(U^{n-1}_1)$,
$$
\lambda^{n-1}_{-} =(1-\alpha^{n-1})\lambda_{-}(U^{n-1}_0)+ \alpha^{n-1}  \lambda_{-}(U^{n-1}_1)
$$

and $0 \leq \alpha^{n-1}  \leq 1$ computed such that $\lambda^{n-1}_{-}  \, \delta_t = \alpha^{n-1} \, \delta_x  $.
 Therefore
\begin{equation}
R_{-,0}^n=(1-\lambda^{n-1}_{-}  \frac{\delta_t}{\delta_x})R_{-,0}^{n-1}  + \lambda^{n-1}_{-} \frac{\delta_t}{\delta_x}R_{-,1}^{n-1}, \label{charac}
\end{equation}
which gives $R_{-,0}^n$ in terms of its values at the previous time step and in terms of interior points.\\


\section{The  Boussinesq equations}\label{sectBou}

We consider here the following Boussinesq-Abbott system \cite{Abbott,FBCR}, which includes the dispersive effects neglected by the nonlinear shallow water equations \eqref{WE1ter}
\begin{equation}\label{Bouss1ND}
\begin{cases}
\dt \zeta +\dx q =0,\\
(1-\frac{H_0^2}{3}\dx^2)\dt q+\dx \big(\frac{1}{2} g h^2 + \frac{1}{h}q^2\big)=0,\qquad (h=H_0+\zeta),
\end{cases}
\end{equation}
complemented with the initial and boundary conditions
\begin{equation}\label{IBC_Bouss}
(\zeta,q)(t=0,x)=(\zeta^0,q^0)(x),\qquad
\zeta(t,x=0)=f(t).
\end{equation}
The key step in our analysis is the reformulation of this IBVP into a system of two conservation laws with nonlocal flux and a source term accounting for the presence of a dispersive boundary layer, and whose coefficient is found through the resolution of a nonlinear ODE.\\
In order to make clearer the structure of the dispersive boundary layer, we work with a dimensionless version of \eqref{Bouss1ND_int}. The non-dimensionalization is performed in \S \ref{sectdim}. The reformulation of the equations is then derived in \S \ref{sectreform} and a numerical scheme based on this newly exhibited structure is proposed in \S \ref{sectdisc}.

\subsection{Dimensionless equations}\label{sectdim}

Denoting by $a$ the typical amplitude of the waves, by $L$ its typical horizontal scale,  we introduce the following dimensionless quantities, denoted with a prime,
$$
x'=\frac{x}{L}, \quad t'=\frac{t}{L/\sqrt{gH_0}}, \quad \zeta'=\frac{\zeta}{a},\quad u'=\frac{u}{\frac{a}{H_0}\sqrt{gH_0}},\quad h'=1+\eps \zeta'.
$$
Replacing in \eqref{Bouss1ND} (and omitting the primes for the sake of clarity), we obtain the dimensionless version of the Boussinesq equations
\eqref{WE1ter}
\begin{equation}\label{WE1terND}
\begin{cases}
\dt \zeta +\dx q =0,\\
(1-\frac{\mu}{3}\dx^2)\dt q+\dx \big(\frac{1}{2\eps} h^2 + \eps\frac{1}{h}q^2\big)=0,
\end{cases}
\end{equation}
where $\eps$ and $\mu$ are respectively called {\it nonlinearity} and {\it shallowness} parameters and defined as
$$
\eps=\frac{a}{H_0},\qquad \mu=\frac{H_0^2}{L^2};
$$
the Boussinesq equations are derived in the shallow water, weakly nonlinear regime characterized by the assumptions
\begin{equation}\label{asssmall}
\mu \ll 1\quad\mbox{ and }\quad \eps=O(\mu).
\end{equation}
Under these smallness assumptions, the Boussinesq model \eqref{WE1terND} provides an approximation consistent with the full free surface Euler equations up to $O(\mu^2)$ and the convergence error is of order $O(\mu^2 t)$ for times of order $O(1/\eps)$ \cite{AL,Lannes_book,Lannes_SW}.  

\subsection{Reformulation of the equations}\label{sectreform}

Solving the equations \eqref{WE1terND} on the full line requires the inversion of the operator $(1-\frac{\mu}{3}\dx^2)$, which does not raise any difficulty. The situation is different here since we need to invert this operator on the half-line $(0,\infty)$, and we therefore need a boundary condition on $\dt q$ which is not directly at our disposal. Our strategy is, as in \cite{BLM} for the description of the interaction of a floating objects with waves governed by a Boussinesq model, to use the inverse of  the operator $(1-\frac{\mu}{3}\dx^2)$ with homogeneous Dirichlet boundary condition, and to construct the dispersive boundary layer due to the fact that  the boundary value $\underline{q}$  of $q$ is not equal to zero in general; we shall denote
$$
\underline{q}(t)=q(t,x=0),
$$
and we also need to define the Dirichlet and Neumann inverses of the operator $(1-\frac{\mu}{3}\dx^2)$.
\begin{definition}\label{defR}
We denote by $R_0$ and $R_1$ the inverse of the operator $(1-\frac{\mu}{3}\dx^2)$ with homogeneous Dirichlet and Neumann boundary conditions respectively,
$$
R_0:\begin{array} {lcl}
L^2(\R_+) &\to & H^2(\R_+)\\
g &\mapsto &  u,
\end{array}
\quad\mbox{ where }\quad
\begin{cases} (1-\frac{\mu}{3}\dx^2)u=g, \\
u(0)=0, \end{cases}
$$
and
$$
R_1:\begin{array} {lcl}
L^2(\R_+) &\to & H^2(\R_+)\\
g &\mapsto &  v,
\end{array}
\quad\mbox{ where }\quad
\begin{cases} (1-\frac{\mu}{3}\dx^2)v=g,\\
\partial_x v(0)=0. \end{cases}
$$
We also introduce the boundary operator $\underline{R}_1$ as
$$
\underline{R}_1:\begin{array} {lcl}
L^2(\R_+) &\to & \R\\
g &\mapsto &  (R_1 g)_{\vert_{x=0}}.
\end{array}
$$
\end{definition}
Recalling that the ODE
$$
Y-\frac{\mu}{3}Y''=g, \qquad Y(0)=Y_0
$$
admits a unique solution in $H^2(\R_+)$ given by
$$
Y(x)=(R_0 g)(x) +Y_0 \exp\big(-\frac{x}{\delta}\big)
\quad \mbox{ with }\quad
\delta=\sqrt{\frac{\mu}{3}},
$$
the second equation in \eqref{WE1terND} can be written equivalently under the form
\begin{equation}\label{reform1}
\dt q=-R_0 \dx \big( \frac{1}{2\eps}h^2+\eps\frac{1}{h}q^2 \big)+ \dot{\underline{q}}\exp\big(-\frac{x}{\delta}\big).
\end{equation}
The last step is therefore to express $\dot{\underline{q}}$ in terms of the data $f={\zeta}_{\vert_{x=0}}$ of the problem. This is done in the following proposition.
\begin{proposition}\label{prop1}
If $(\zeta,q)$ are a smooth enough solution of  \eqref{WE1terND}, then the boundary value $\underline{q}$ of $q$ are related to  the boundary value $f=\zeta_{\vert_{x=0}}$ and to the interior value of $\zeta$ and $q$ by solving the ODE
$$
\dot{\underline{q}}-\frac{\varepsilon}{\delta}\frac{\underline{q}^2}{1+\varepsilon f}=\delta\ddot f+\frac{1}{\delta}(1+\frac{\eps}{2}f)f- \frac{1}{\delta} \underline{R}_1 \big( \frac{1}{2\eps}(h^2-1)+\eps\frac{1}{h}q^2 \big),
$$
where  $\underline{R}_1$ is the boundary operator introduced in Definition \ref{defR}.
\end{proposition}
\begin{proof}
Differentiating \eqref{reform1} with respect to $x$, one obtains
\begin{equation}\label{reform2}
\dt \dx q=-\dx R_0 \dx \big( \frac{1}{2\eps}(h^2-1)+\eps\frac{1}{h}q^2 \big)- \frac{1}{\delta}\dot{\underline{q}}\exp\big(-\frac{x}{\delta}\big).
\end{equation}
\begin{lemma}
For all $g \in L^2(\R_+)$, the following identity holds,
$$
R_0 \dx g=\dx R_1 g.
$$
\end{lemma}
\begin{proof}[Proof of the lemma]
Just remark that if $v=R_1 g$, then one easily gets from the definition of $R_1$ that
$$
\begin{cases}
(1-\frac{\mu}{3}\dx^2) (\dx v)=\dx g,\\
(\dx v)(0)=0,
\end{cases}
$$
so that, by definition of $R_0$, one has $\dx v=R_0 \dx g$ (note that by classical variational arguments, $R_0$ is well defined as a mapping $\dx L^2(\R_+)\to H^1(\R_+)$).
\end{proof}
Using  the first equation of \eqref{WE1terND} to substitute $\dt\dx q=-\dt^2 \zeta$ and the lemma, one then deduces from \eqref{reform2} that
$$
-\dt^2 \zeta=-\dx^2 R_1  \big( \frac{1}{2\eps}(h^2-1)+\eps\frac{1}{h}q^2 \big)- \frac{1}{\delta}\dot{\underline{q}}\exp\big(-\frac{x}{\delta}\big).
$$
Remarking further that $-\dx^2=\frac{1}{\delta^2}(1-\frac{\mu}{3}\dx^2)-\frac{1}{\delta^2}$ and recalling that $(1-\frac{\mu}{3}\dx^2)R_1 =\mbox{Id}$, we obtain that
$$
\dt^2 \zeta= \frac{1}{\delta^2}(R_1-\mbox{Id})  \big( \frac{1}{2\eps}(h^2-1)+\eps\frac{1}{h}q^2 \big)+ \frac{1}{\delta}\dot{\underline{q}}\exp\big(-\frac{x}{\delta}\big).
$$
Taking the trace of this expression at $x=0$ then yields
$$
\ddot{f}+\frac{1}{\delta^2}(1+\frac{\eps}{2}f)f= \frac{1}{\delta^2}\big[R_1 \big( \frac{1}{2\eps}(h^2-1)+\eps\frac{1}{h}q^2 \big)\big]_{\vert_{x=0}}+ \frac{1}{\delta}\dot{\underline{q}}-\frac{\eps}{\delta^2}\frac{\underline{q}^2}{1+\varepsilon f},
$$
from which the result follows.
\end{proof}

Using once again the lemma to replace $R_0\dx$ by $\dx R_1$ in \eqref{reform1}, it follows from the above that the dimensionless Boussinesq equations \eqref{WE1terND} with initial and boundary conditions \eqref{IBC_Bouss} can be equivalently written under the form
\begin{equation}\label{WE1terND_reform}
\begin{cases}
\displaystyle \dt \zeta +\dx q =0,\\
\displaystyle \dt q+\dx R_1  {\mathfrak f}(\zeta,q)= \underline{{\mathcal Q}}(\underline{q},f,\ddot f, \zeta, q)\exp\big(-\frac{x}{\delta}\big),\\
\end{cases}
\end{equation}
where $\underline{q}=q_{\vert_{x=0}}$ and ${\mathfrak f}(\zeta,q)$ is the flux in the momentum equation for the nonlinear shallow water equations \eqref{WE1ter} in dimensionless variables,
$$
{\mathfrak f}(\zeta,q):=\frac{1}{2\eps}(h^2-1)+\eps \frac{1}{h}q^2,
$$
and
\begin{equation}\label{defODE}
\underline{{\mathcal Q}}(\underline{q},f,\ddot f, \zeta, q)=\frac{\varepsilon}{\delta}\frac{\underline{q}^2}{1+\eps f}+\delta \ddot f+\frac{1}{\delta}(1+\frac{\eps}{2} f)f- \frac{1}{\delta} \underline{R}_1 {\mathfrak f}(\zeta,q),
\end{equation}
with the initial condition
\begin{equation}\label{IBC_Bouss_bis}
(\zeta,q)(t=0,x)=(\zeta^0,q^0)(x)
\end{equation}
and the boundary condition
\begin{equation}\label{Bouss_reformBC}
\zeta(t,x=0)=f(t).
\end{equation}
\begin{remark}
Recalling that by definition of $R_1$, the trace of $\dx R_1 {\mathfrak f}$ vanishes at $x=0$, one can take the trace at $x=0$ in the second equation in \eqref{WE1terND_reform} to obtain the following evolution equation on $\underline{q}=q_{\vert_{x=0}}$,
\begin{equation}\label{qdisp}
\displaystyle \dot{\underline{q}}=\underline{{\mathcal Q}}(\underline{q},f,\ddot f, \zeta, q).
\end{equation}
This relation has to be compared to \eqref{qhyp} in the hyperbolic case, where $\underline{q}$ is given in terms of $f=\zeta_{\vert_{x=0}}$ and the trace of the outgoing Riemann invariant $R_-$. The mechanisms that allow to express $\underline{q}$ in terms of $f$ and interior values of $\zeta$ and $q$ are therefore completely different in the hyperbolic and in the dispersive cases: in the former, the decomposition into Riemann invariants is used to propagate information from the interior domain, while in the latter, this is done by using the non local nature of the operator $R_1$.
\end{remark}
\subsection{Well-posedness of the initial boundary value problem}\label{sectWP}

As said in the introduction, very few results exist regarding the local well-posedness result for Boussinesq systems, except in some special cases such as \cite{ADM,Xue}. To our knowledge, no result exist so far for the Boussinesq-Abbott system considered here. Our reformulation \eqref{WE1terND_reform}-\eqref{IBC_Bouss_bis} of this IBVP allows an easy proof of local well-posedness since it forms a simple ODE on $(\zeta,q)$ (here again, this is in strong contrast with the hyperbolic case where, of course,  the equations cannot be put recast as an ODE).
\begin{theorem}
Let $f\in C^2(\R_+)$, $n\in {\mathbb N}\backslash\{0\}$ and $(\zeta^0,q^0)\in H^n(\R_+)\times H^{n+1}(\R_+)$ be such that $\inf (1+\eps \zeta^0)>0$. Then there exist $T>0$ and a unique solution $(\zeta,q)\in C^1([0,T]; H^n(\R_+)\times H^{n+1}(\R_+))$ to \eqref{WE1terND_reform}-\eqref{IBC_Bouss_bis}. \\
If moreover ${\zeta^0}_{\vert_{x=0}}=f(0)$ and $-{\dx q^0}_{\vert_{x=0}}=\dot f(0)$, then the boundary condition \eqref{Bouss_reformBC} is also satisfied for all times.
\end{theorem}
\begin{remark}
The existence time furnished by the theorem depends on $\eps$ and $\mu$. The relevant time scale for the existence of the solution is $O(1/\eps)$ in the case of the full line \cite{AL,Lannes_book}. Proving that such a time scale is also achieved in our case would require much more effort and an in depth analysis of the dispersive boundary layer together with additional compatibility conditions. Such a study is performed in \cite{BLM} in the related problem of waves interaction with a floating object in the Boussinesq regime.
\end{remark}
\begin{proof}
To prove the first part of the theorem, it is enough to prove that \eqref{WE1terND_reform}-\eqref{IBC_Bouss_bis} is actually an ODE on $H^n(\R^+)\times H^{n+1}(\R^+)$ meeting the requirements of the Cauchy-Lipschitz theorem. \\
With $U=(\zeta,q)^{\rm T}$ we can write the equations under the form
$$
\dt U={\mathcal F}(t,U)
\quad\mbox{ with }\quad
{\mathcal F}(t,U)=\left(\begin{array}{c}
-\dx q\\
-\dx R_1{\mathfrak f}(\zeta,q)+\underline{\mathcal Q}(\underline{q},f,\ddot f, \zeta, q)\exp(-\frac{x}{\delta})\\
\end{array}\right).
$$
By standard product estimates, $(\zeta,q)\in H^n\times H^{n+1}\mapsto {\mathfrak f}(\zeta, q)\in H^n$ is regular in a neighborhood of $(\zeta^0,q^0)$; moreover, $\dx R_1$ maps $H^n$ into $H^{n+1}$ by definition of $R_1$. It follows easily that ${\mathcal F}(t,U)$ is continuous and locally Lipschitz with respect to the second variable, so that we can apply Cauchy-Lipschitz theorem.\\
We now need to check that  $\zeta(t,0)=f(t)$ for all time. In order to do so, one computes from the first equation in  \eqref{WE1terND_reform} that $\dt^2 \zeta=-\dt \dx q$. Using the second equation to compute $\dt \dx q$ and taking the trace at $x=0$ one gets (proceeding as in the proof of Proposition \ref{prop1}) that
$$
\frac{d^2}{dt^2} (\zeta_{\vert_{x=0}})=\ddot f +\frac{1}{\delta^2} \Big( {\mathfrak f}(f,\underline{q})-{\mathfrak f}(\zeta_{\vert_{x=0}},\underline{q})\Big).
$$
This can be seen as a second order non-autonomous ODE on $\zeta_{\vert_{x=0}}$ with a right-hand side that is locally Lipschitz with respect to  $\zeta_{\vert_{x=0}}$. There is therefore a unique solution to this ODE satisfying the initial conditions $\zeta_{\vert_{x=0}}(0)=f(0)$ and $\frac{d}{dt}(\zeta_{\vert_{x=0}})(0)=-\dx q^0 (0)=\dot f(0)$. This solution is obviously given by $\zeta_{\vert_{x=0}}=f$, so that the proof is complete.
\end{proof}
\subsection{Discretization of the equations}\label{sectdisc}

The goal of this section is to derive a numerical scheme to solve the initial boundary value problem \eqref{WE1terND_reform}-\eqref{Bouss_reformBC}.

\subsubsection{A discrete version of the operators $R_1$ and $\underline{R}_1$}\label{sectdiscR}

We still denote by $R_1$ the discrete inverse of the operator $(1-\frac{\mu}{3}\dx ^2)$ with homogeneous Neumann condition at the boundary. We use here a standard centered second order finite difference approximation for the discretization of $\dx^2$. More precisely, if $F=(f_i)_{i\geq 1}$, we denote by $R_1 F$ the vector $R_1 F=V$ where $V=(v_i)_{i\geq 1}$ is given by the resolution of the equations
$$
v_i -\frac{\mu}{3}\frac{v_{i+1}-2 v_i +v_{i-1}}{\delta_x^2}=f_i,\qquad i\geq 2
$$
while, for $i=1$ the Neumann boundary condition is taken into account as follows,
$$
v_1 -\frac{\mu}{3}\frac{v_{2}- v_1}{\delta_x^2}=f_1.
$$
Similarly, we still denote by $\underline{R}_1$ the discrete version of the boundary operator $\underline{R}_1$, naturally defined by the second order approximation
$$
\underline{R}_1 F=v_1.
$$
\subsubsection{A finite volume scheme with nonlocal flux}

We first rewrite  \eqref{WE1terND_reform}  in the condensed form
\begin{equation}\label{WE2_B}
\dt U+\dx \big({\mathfrak F}_\mu(U) \big)=S
\end{equation}
with $U=(\zeta,q)^T$ and 
\begin{equation}\label{deftzeta}
{\mathfrak F}_\mu(U)=\big(q, {\mathfrak f}_\mu(U)\big)^T,
\end{equation}
and where
$$
{\mathfrak f}_\mu(U)=R_1 {\mathfrak f}(U)
\quad\mbox{ and }\quad
{\mathfrak f}(U):=\frac{1}{2\eps}(h^2-1)+\eps \frac{1}{h}q^2,
$$
(${\mathfrak f}(U)$ is the flux in the momentum equation for the nonlinear shallow water equations \eqref{WE1ter} in dimensionless variables). The flux in \eqref{WE2_B} is therefore a nonlocal operator with respect to $U$. The source term $S$ in \eqref{WE2_B} is given by
\begin{equation}\label{source}
S=\left(\begin{array}{cc} 0 \\
 \underline{{\mathcal Q}}(\underline{q},f,\ddot f, \zeta, q)\exp(-\frac{x}{\delta})
 \end{array}\right),
\end{equation}
where we recall that $\underline{{\mathcal Q}}(\underline{q},f,\ddot f, \zeta, q)$ is defined in \eqref{defODE}.\\
Using a  finite volume type discretization for the \eqref{WE2_B} and a standard Euler scheme for the ODE on $\underline{q}$, we obtain the following general discretization of the Boussinesq system \eqref{WE1terND_reform},
\begin{equation}\label{schema}
\displaystyle \frac{U_i^{n+1}-U_i^n}{\delta_t}+\frac{1}{\delta_x}({\mathfrak F}^n_{\mu,i+1/2}-{\mathfrak F}^n_{\mu,i-1/2})=S^n_i, \qquad i\geq 1,\quad n\geq 0,
\end{equation}
where $U^n=(\zeta^n,q^n)^{ T}=(\zeta_i^n,q_i^n)_{i\geq 1}^{ T}$ and the source term $S^n_i$ is being given by
\begin{equation}\label{source_disc}
S^n_i=\left(\begin{array}{cc} 0 \\
 \underline{{\mathcal Q}}(q_0^n,f^n,\ddot f^n, \zeta^n, q^n)\exp(-\frac{x_i}{\delta})
 \end{array}\right)
\end{equation}
(note that the definition for the discretized version of $\underline{\mathcal Q}$ can straightforwardly be deduced from \eqref{defODE} along the lines of \S \ref{sectdiscR}; see also Remark \ref{remQ} below). The source term involves the quantity $q_0^n$ which cannot be computed by induction through \eqref{schema} since in \eqref{schema}, one assumes that $i\geq 1$. However, a direct discretization of \eqref{qdisp} yields
\begin{equation}\label{schematrace}
\displaystyle \frac{{q}_0^{n+1}-{q}_0^n}{\delta_t}= \underline{{\mathcal Q}}({q}_0^n,f^n,\ddot f^n, \zeta^n, q^n), \qquad n\geq 0.
\end{equation}
It remains of course to explain how to compute the discrete fluxes ${\mathfrak F}_{\mu,i+1/2}$. As above for the nonlinear shallow water equations, we consider here the simplest case of  the Lax-Friedrichs scheme where the discrete flux is given by
\begin{equation}\label{tLF}
{\mathfrak F}^n_{\mu,i-1/2}=\frac{1}{2}({\mathfrak F}^n_{\mu,i}+{\mathfrak F}^n_{\mu,i-1})-\frac{\delta_t}{2\delta_x}(U_i^n-U_{i-1}^n),
\end{equation}
where we write, when $i\geq 1$,
\begin{equation}\label{deffmu}
{\mathfrak F}^n_{\mu,i}=\big(q^n_i,{\mathfrak f}^{n}_{\mu,i}\big)^T\quad\mbox{ with }\quad {\mathfrak f}^n_{\mu}:=R_1 \big({\mathfrak f}(U^n_i)\big)_{i\geq 1},
\end{equation}
the discrete operator $R_1$ being constructed as in \S \ref{sectdiscR}.\\
When $i=0$, this definition is naturally adapted as follows, 
\begin{equation}\label{deffmu1}
{\mathfrak F}^n_{\mu,0}=\big({q}_0^n,{\mathfrak f}^{n}_{\mu,0}\big)^T\quad\mbox{ with }\quad {\mathfrak f}^n_{\mu,0}:=\underline{R}_1 \big({\mathfrak f}(U^n_i)\big)_{i\geq 1},
\end{equation}
the discrete boundary operator $\underline{R}_1$ being constructed as in \S \ref{sectdiscR} while ${q}_0^n$ is provided by  \eqref{schematrace}.

\begin{remark}\label{remQ}
The quantity $ \underline{{\mathcal Q}}(q_0^n,f^n,\ddot f^n, \zeta^n, q^n)$ that appears in the right-hand side of the momentum equation  in \eqref{schema} and in the discrete ODE \eqref{schematrace} for $q_0^n$  can be written using the notation \eqref{deffmu1} as
$$
\underline{{\mathcal Q}}(q_0^n,f^n,\ddot f^n, \zeta^n, q^n)=\frac{\varepsilon}{\delta}\frac{(q_0^n)^2}{1+\eps f^n}+\delta \ddot f^n+\frac{1}{\delta}(1+\frac{\eps}{2} f^n)f^n- \frac{1}{\delta} {\mathfrak f}_{\mu,0}^n.
$$
All these quantities are already known so that handling generating boundary condition can be done with no extra computational cost compared to, say, periodic boundary conditions.
\end{remark}

\section{Numerical validations}\label{sectnumval}

\subsection{The validation method}\label{sectvalmet}

Since the implementation of reflecting or periodic boundary conditions does not raise any problem for the Boussinesq equations \eqref{WE1terND} we compute first a solution $U^L=(\zeta^L,q^L)^T$ of the equations under consideration in a larger domain $[-L,L]$ until a final time $T_f$, with reflective  or periodic boundary conditions at both extremities, and with a non trivial initial condition. We then define a reference solution as the restriction of $U^L$ on $[0,L]$, and a boundary data $f$ as
$$
U^{\rm ref}=(\zeta^{\rm ref},q^{\rm ref})^T:= U^L_{\vert_{[0,L]}}
\quad\mbox{ and }\quad 
f(t):=\zeta^L(t,x=0).
$$
We then use the scheme presented in \S \ref{sectdisc} to compute the solution $U$ of the Boussinesq system \eqref{WE1terND} with initial data $U^0(x)=U^{\rm ref}(t=0,x)$ and boundary data $f$, and compare it with the reference solution $U^{\rm ref}$. 
We define in particular the errors ${\mathbf e^{\zeta}_{\delta_x}}(t)$ and ${\mathbf e^{q}_{\delta_x}}(t)$ as
\begin{equation}\label{approxerr}
{\mathbf e^{\zeta}_{\delta_x}}(t)=\Vert \zeta(t,\cdot)-\zeta^{\rm ref}(t,\cdot)\Vert_{L^\infty(0,L)}, 
\qquad
{\mathbf e^q_{\delta_x}}(t)=\Vert q(t,\cdot)-q^{\rm ref}(t,\cdot)\Vert_{L^\infty(0,L)},
\end{equation}
and we compute the overall errors $e^{\zeta}_{\delta_x}$ and $e^{q}_{\delta_x}$ on $[0,T_f]$ as
$$
e^{\zeta}_{\delta_x} = \Vert  {\bf e}^{\zeta}_{\delta_x}(.)\Vert_{L^\infty(0,T_f)}, \quad e^q_{\delta_x} = \Vert  {\bf e}^q_{\delta_x}(.)\Vert_{L^\infty(0,T_f)}. 
$$
The convergence order $p$ is computed with a least-squares linear regression, whose coefficient is plotted on the error curves.
%
%
%

\subsection{Propagation of gaussian initial conditions}\label{refined}

We recall that the Boussinesq equations \eqref{WE1terND} are derived under the smallness assumption \eqref{asssmall} on $\eps$ and $\mu$. We consider here the approximation error in  different cases,
$$
({\rm I})\quad  \eps=\mu=0.3, \qquad ({\rm II}) \quad  \eps=\mu=0.1,  \quad ({\rm III})\quad  \eps=\mu=0.01,
$$
the nonlinear and dispersive effect become more important when $\eps$ and $\mu$ respectively become larger; in particular, the configuration (${\rm I}$) is quite stiff and in the limit of the range of validity of the Boussinesq equations (for strong nonlinearities, one should rather work with the more complicated Serre-Green-Naghdi equations \cite{Lannes_book,Lannes_SW}).\\
The initial datum for $U^L$ in the larger domain is
\begin{align}
\zeta^{L}(t=0,x)&= e^{-6(x+0.1L)^2}+ e^{-6(x-0.3L)^2}; \label{CIzeta} \\
q^L(t=0,x) &=  e^{-6(x+0.1L)^2}- e^{-6(x-0.3L)^2}; \label{CIq}
\end{align}

The reference solution is  computed with the Lax-Friedrichs scheme with  non local flux introduced above, on the domain $[-L, L]$, with a very refined mesh: $n_x = 3600$, and a time step $\delta_t = 0.9 \delta_x$ in agreement with the CFL condition computed from the approximated velocities of the Riemann invariants. We take $L=5$.
We compute the numerical solution with the nonlocal Lax-Friedrichs scheme in the domain $[0,L]$, on coarser meshes: $n_x = 90,120,150,180,200,300,360,450$.
 The meshes are defined so that the points of the coarse meshes always coincide with the points of the finer mesh.
The boundary conditions at $x=0$ are taken into account by imposing the reference solution and its second-order time derivative approximated with the classical centered second-order scheme.
 
As the initial data is zero near the boundaries of the large domain, no special effort is necessary for the computation with the coarse mesh at the right boundary $x=L$  if the final time of the simulation is not too large.
We shall compare the solution over a time interval $t\in [0,2]$. The qualitative behavior of the solution is the following: each of the two gaussians decomposes into two waves roughly traveling at speed $1$ and $-1$ respectively. The gaussian located on the left being closer to the boundary $x=0$ of the small domain, this configuration is rich enough to contain the three main relevant cases,
\begin{enumerate}
\item The forcing $f$ corresponds to an essentially incoming wave. This is the situation that occurs for $t\sim 0.1$ (see Figure  \ref{both} left)
\item The forcing $f$ corresponds to the superposition of an outgoing and an incoming wave. This is the situation that occurs for $t\sim 1$ (see Figure  \ref{both} middle)
\item The forcing $f$ corresponds to an essentially outgoing wave. This is the situation that occurs for $t\sim 1.5$ (see Figure \ref{both} right)
\end{enumerate}
 \begin{figure}[!ht]
\centering
\begin{tabular}{ccc}
 \hspace{-25mm}  \includegraphics[height=45mm]{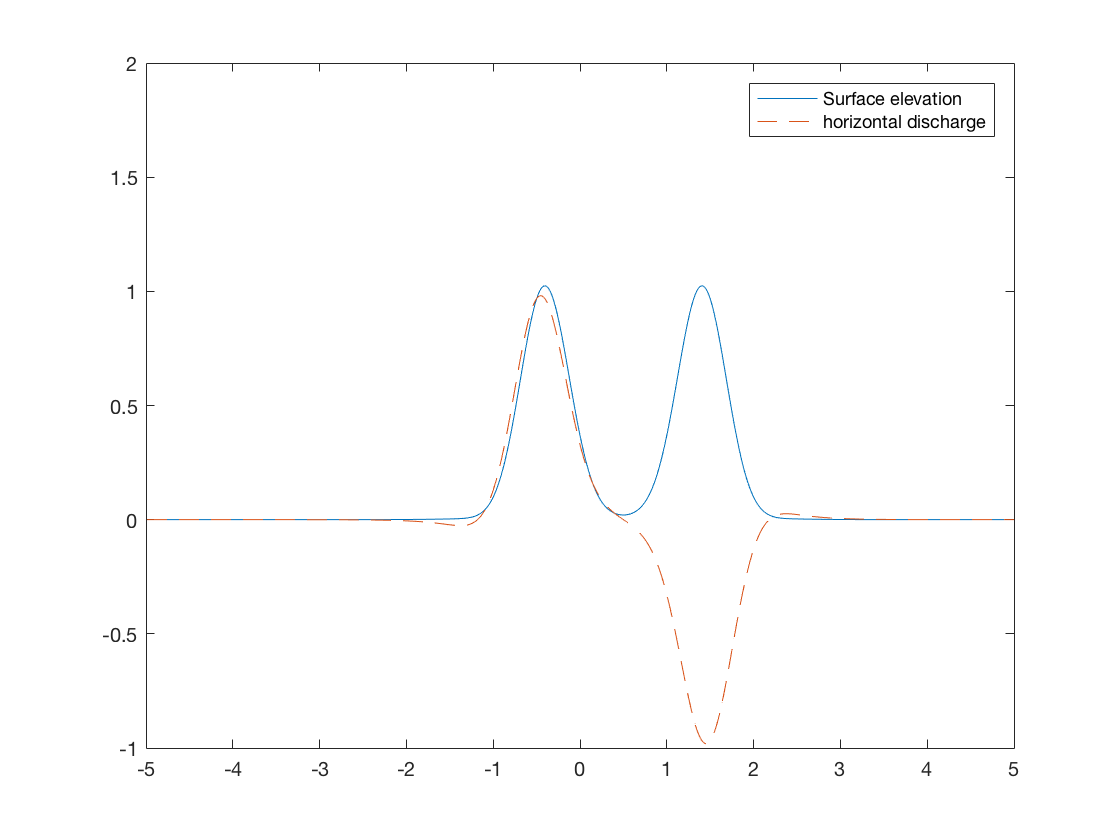}  & \includegraphics[height=45mm]{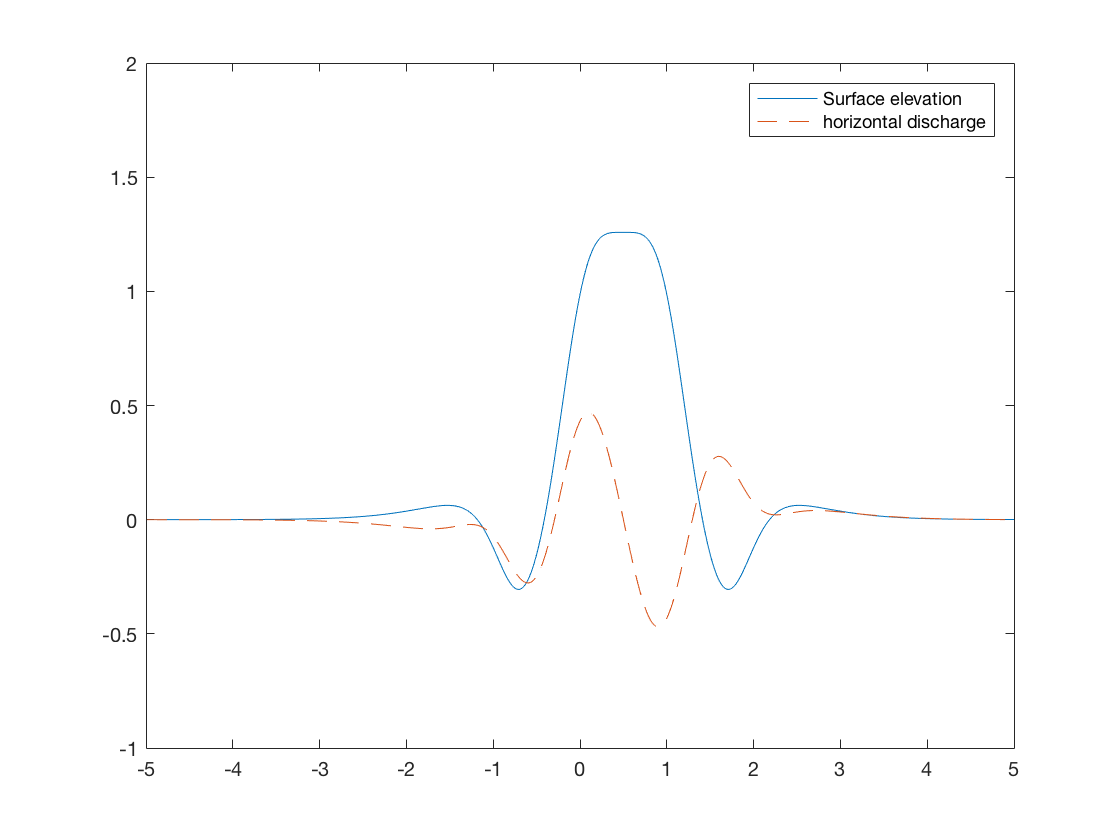} & \includegraphics[height=45mm]{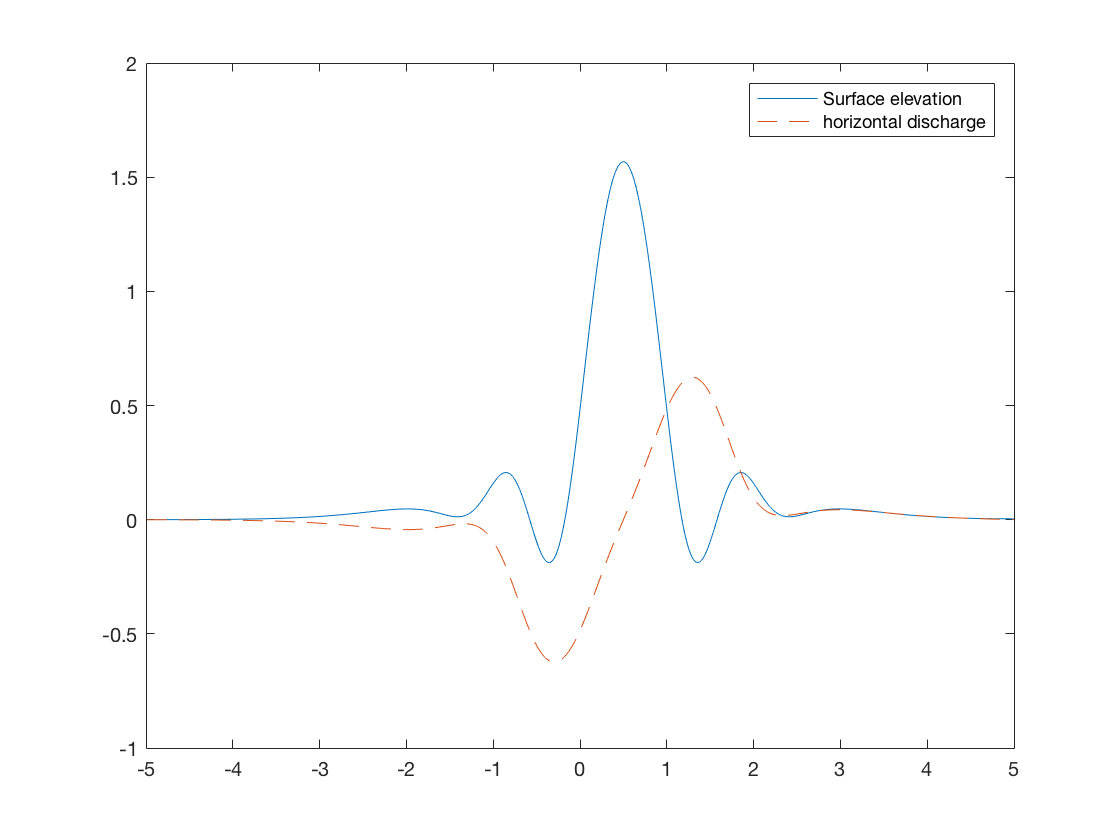}
\end{tabular}
\caption{Numerical results on large domain with $\delta_x = L/400$, with $L = 5$, $\mu = \eps = 0.3$, at times $T = 0.1, 1., 1.5$. } { \label{both}}
   \end{figure}

Numerical results for the initial condition (\ref{CIzeta})-(\ref{CIq}) are presented on Figures \ref{table1}, \ref{table2} and \ref{table3}, in logarithmic scale with the slope obtained from a linear regression. 
On Figure  \ref{table1}, corresponding to the case $\mu = \epsilon = 0.3$, the slope of the linear regression obtained with all error points is completed with the slope of the linear regression obtained with the four more refined error points.
Globally, a first-order convergence in space is observed for both variables. 

 \begin{figure}[!ht]
\centering
\begin{tabular}{cc}
   \includegraphics[height=50mm]{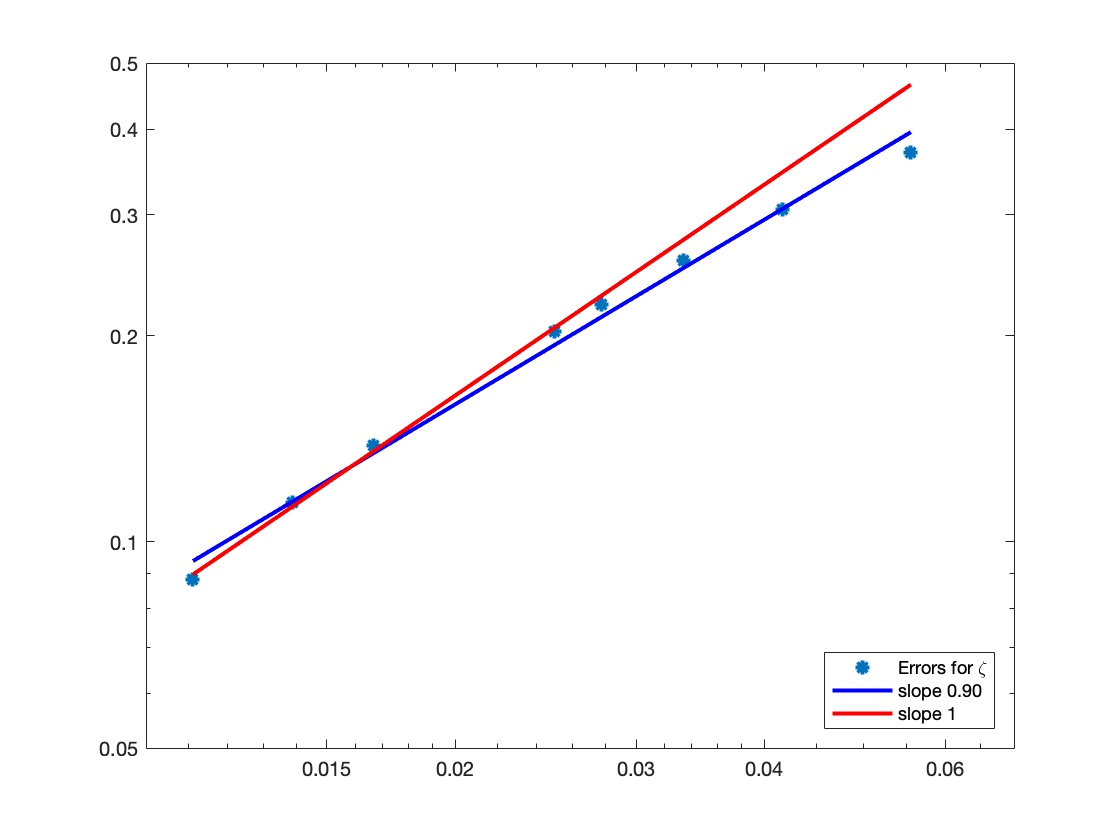}  & \includegraphics[height=50mm]{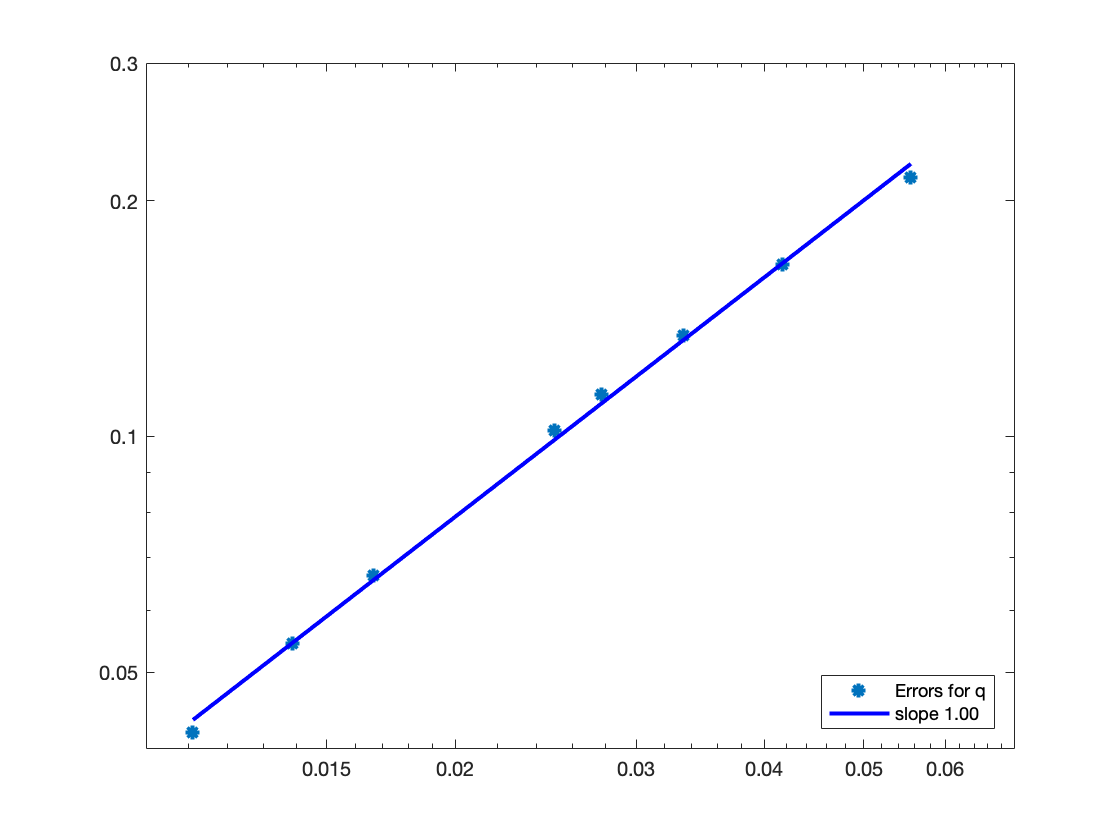}
\end{tabular}
  \caption{Convergence results for Boussinesq equations, $\mu = \eps = 0.3$.}
  \label{table1}
     \end{figure}

 \begin{figure}[!ht]
\centering
\begin{tabular}{cc}
   \includegraphics[height=50mm]{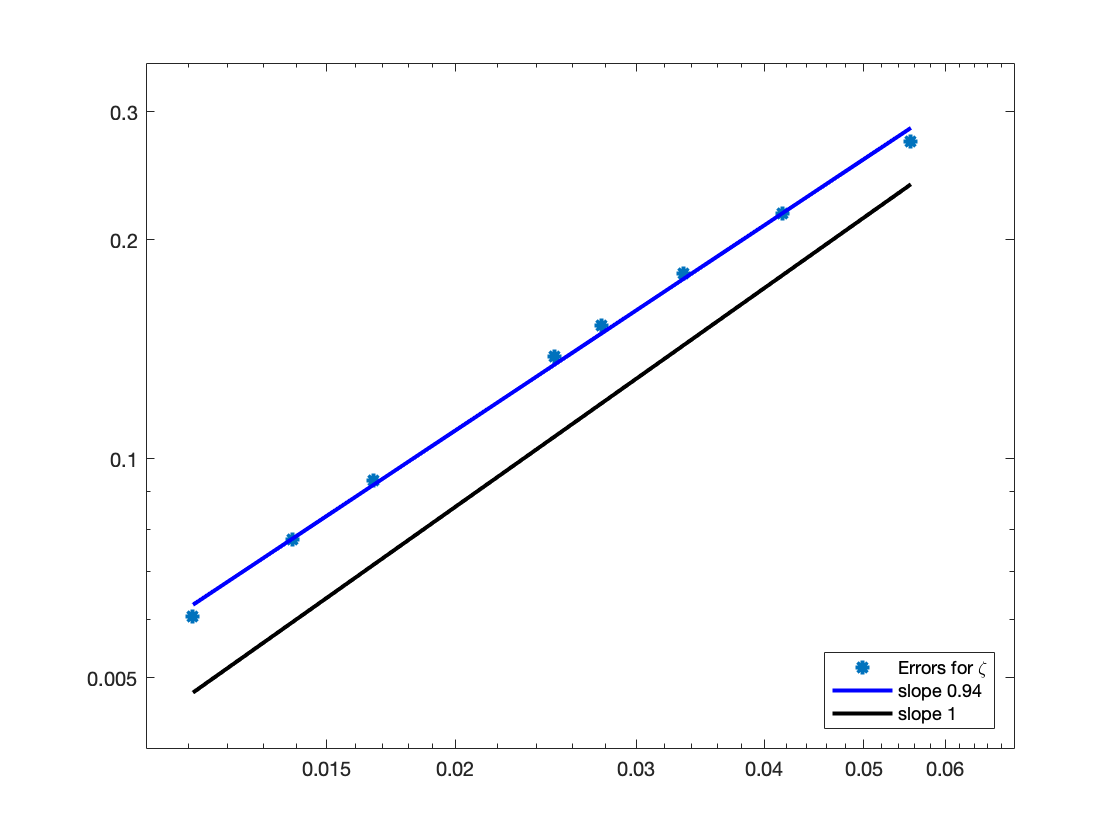}  & \includegraphics[height=50mm]{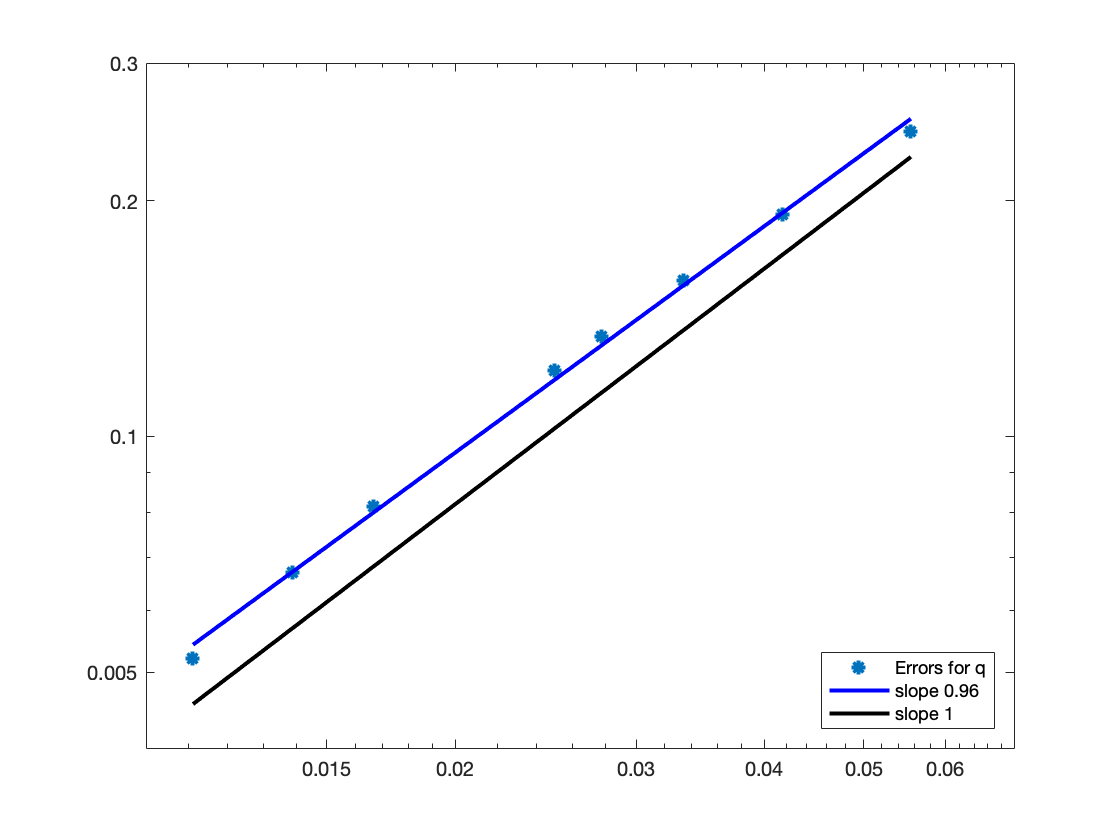}
\end{tabular}
  \caption{Convergence results for Boussinesq equations, $\mu = \eps = 0.1$.}
  \label{table2}
     \end{figure}

 \begin{figure}[!ht]
\centering
\begin{tabular}{cc}
   \includegraphics[height=50mm]{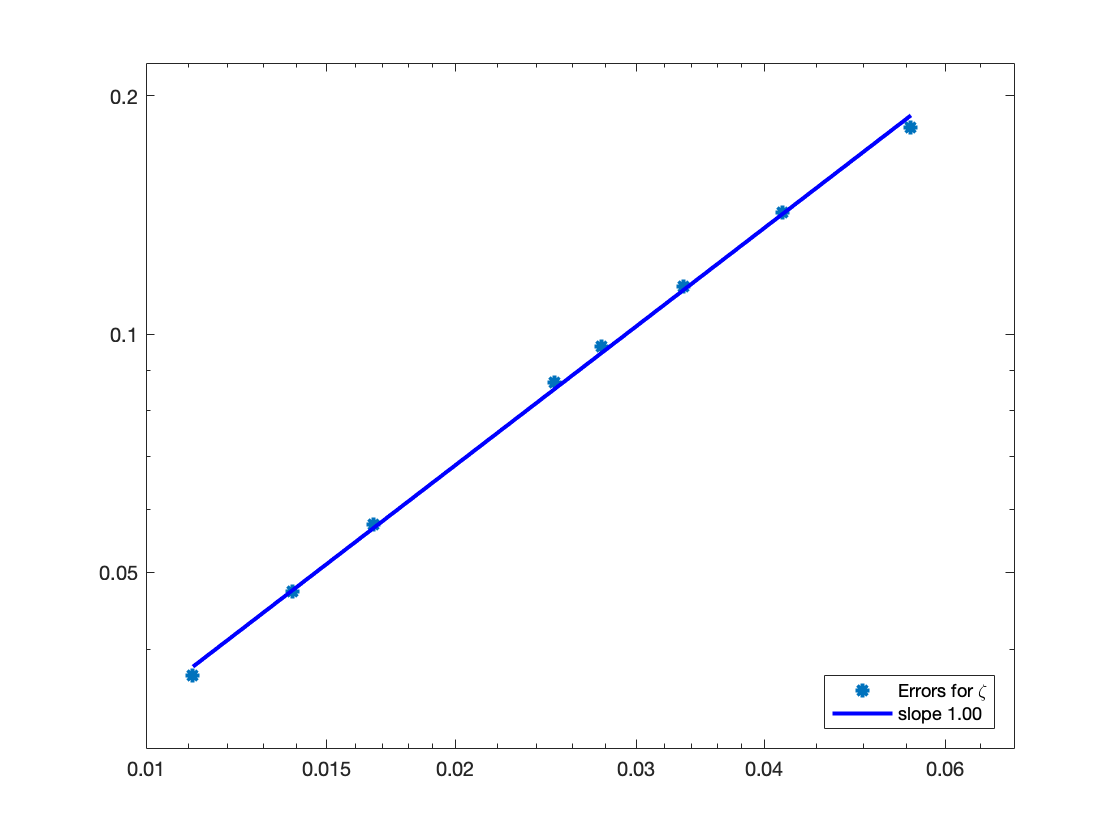}  & \includegraphics[height=50mm]{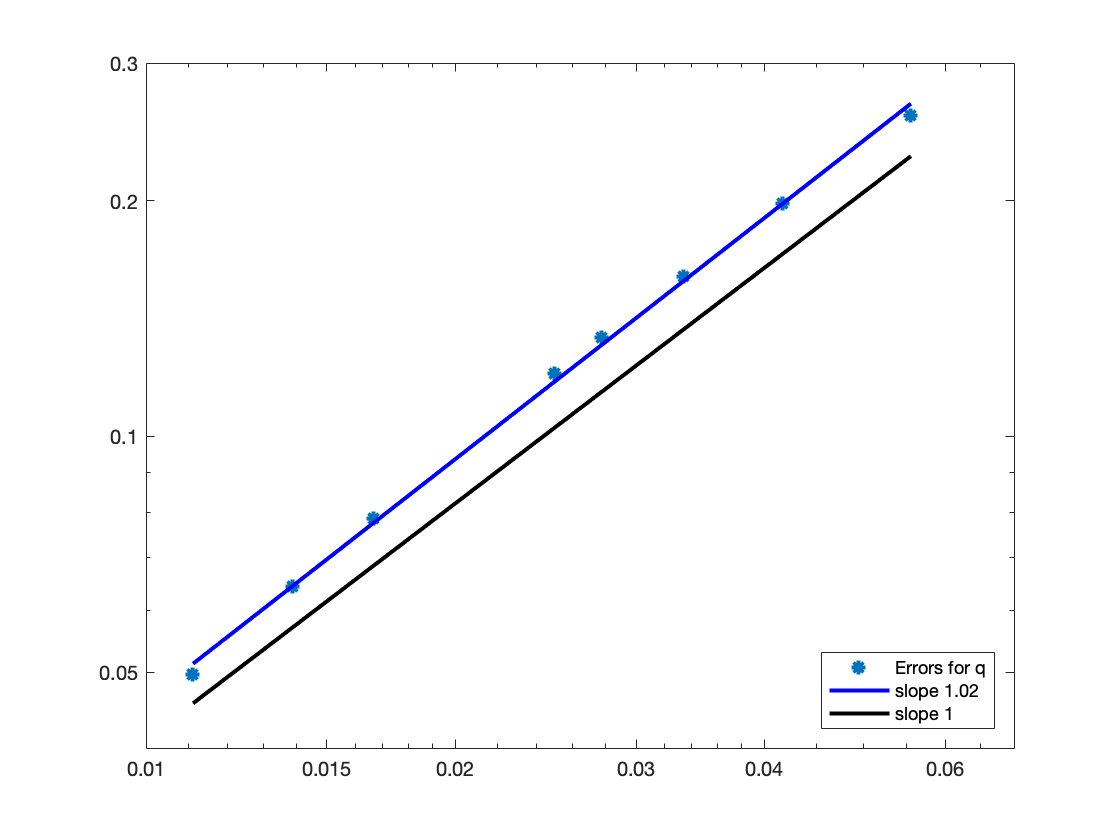}
\end{tabular}
  \caption{Convergence results for Boussinesq equations, $\mu = \eps = 0.01$.}
  \label{table3}
     \end{figure}

\subsection{Soliton propagation}\label{soliton}

We test here our scheme on the propagation of a solitary wave, which involves both nonlinear and dispersive effects.
The soliton for the non-linear Boussinesq system \eqref{WE1terND} is not explicit, but we  compute it by solving numerically a second-order differential equation that we obtain as follows. We look for a solution of the non-linear Boussinesq equations such that $\zeta(x,t) = \tilde{\zeta}(x-ct)$ and $q(x,t) = \tilde{q}(x-ct)$.
We inject this form in the first equation of (\ref{WE1terND})  and find, omitting the tilde symbol  for the sake of brevity
$$
q' = c \, \zeta'.
$$
Then we inject this relationship in the second equation of (\ref{WE1terND}), integrate in space, and we can write (imposing that $\zeta$ vanishes at infinity),
\begin{eqnarray}
- c^2 \frac{\zeta}{ 1 + \eps \zeta}  + \frac{c^2\, \mu}{3}  \zeta'' + \frac{\eps^2 \, \zeta^2+ 2 \, \eps \, \zeta}{2 \eps} &= 0. \label{EDOzeta}
\end{eqnarray}
Multiplying this equation by $\zeta'$ and using again that $\zeta$ tends to zero when $x$ tends to $\pm \infty$, we obtain 
$$
-\frac{c^2}{\eps} \Big(     \zeta - \frac{\ln (1+\eps \zeta) }{\eps} \Big)   + \frac{c^2 \mu}{6}  (\zeta')^2  + \frac{\eps}{2}  \frac{\zeta^3}{3} + \frac{\zeta^2}{2} = 0.
$$
Denoting $\zeta_{{\rm max}}$ the maximum value of $\zeta$ we can compute $c$ as a function of $\zeta_{{\rm max}}$ and $\eps$.
$$
c^2 =  \eps \frac{  \frac{\eps \zeta_{{\rm max}}^3}{6}  + \frac{\zeta_{{\rm max}}^2}{2}  }{  \zeta_{{\rm max}} - \frac{\ln(1+\eps \zeta_{{\rm max}})}{\eps}    }.
$$
 Once $c$ is computed, we solve the differential equation  (\ref{EDOzeta})  with a high order numerical method in order to obtain our reference solution.
 We  choose   $\zeta_{{\rm max}}=1$ and $\mu = \eps = 0.3$ or $0.1$.
We have checked that if we solve the Boussinesq system with this reference solution as an initial datum, with the nonlocal Lax-Friedrichs scheme and periodic boundary conditions, 
the numerical results show a first order convergence: see Figures  \ref{table_soliton1} and \ref{table_soliton2}.
The space steps $\delta_x$ were computed as $\delta_x = L/n_x$, with $n_x = 400, 600, 800, 1000, 1200, 1400, 1600, 1800, 2000$.

\begin{figure}[!ht]
\centering
\begin{tabular}{cc}
   \includegraphics[height=50mm]{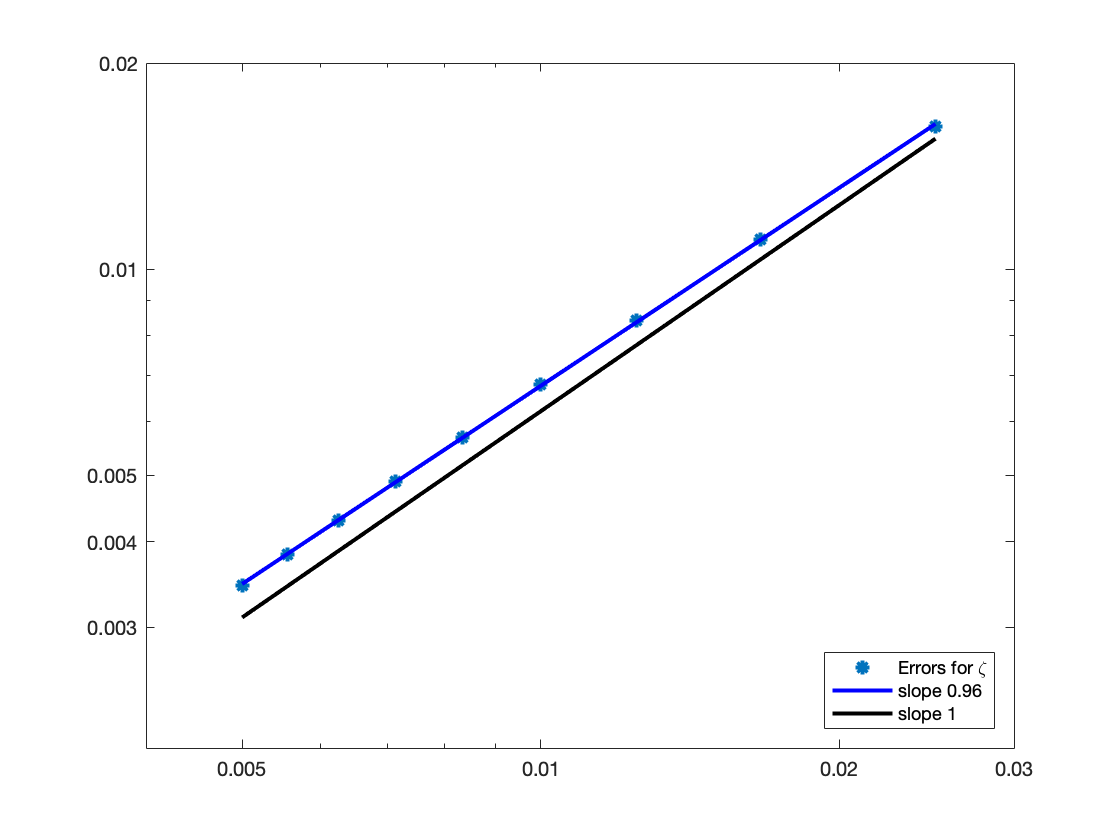}  & \includegraphics[height=50mm]{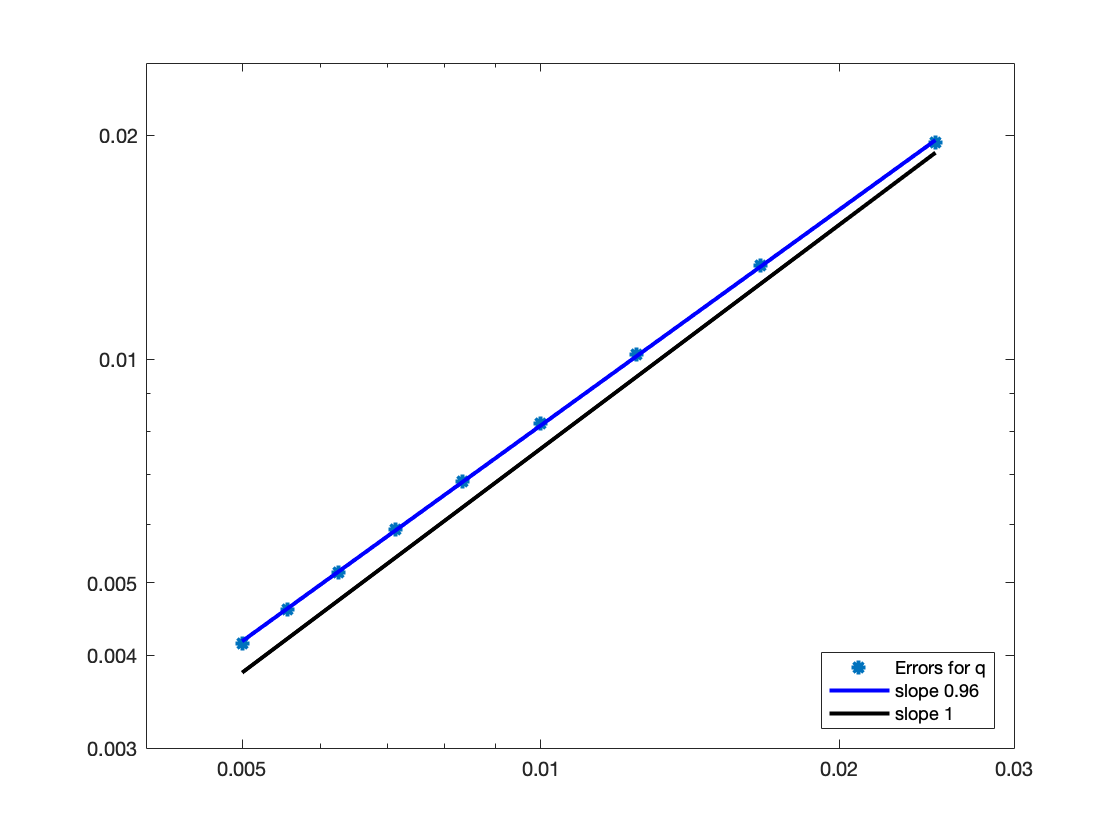}
\end{tabular}
   \caption{Validation of convergence of soliton propagation in the large domain, $L^{\infty}$ error, $\mu = \eps = 0.3$.}
  \label{table_soliton1}
     \end{figure}

 \begin{figure}[!ht]
\centering
\begin{tabular}{cc}
   \includegraphics[height=50mm]{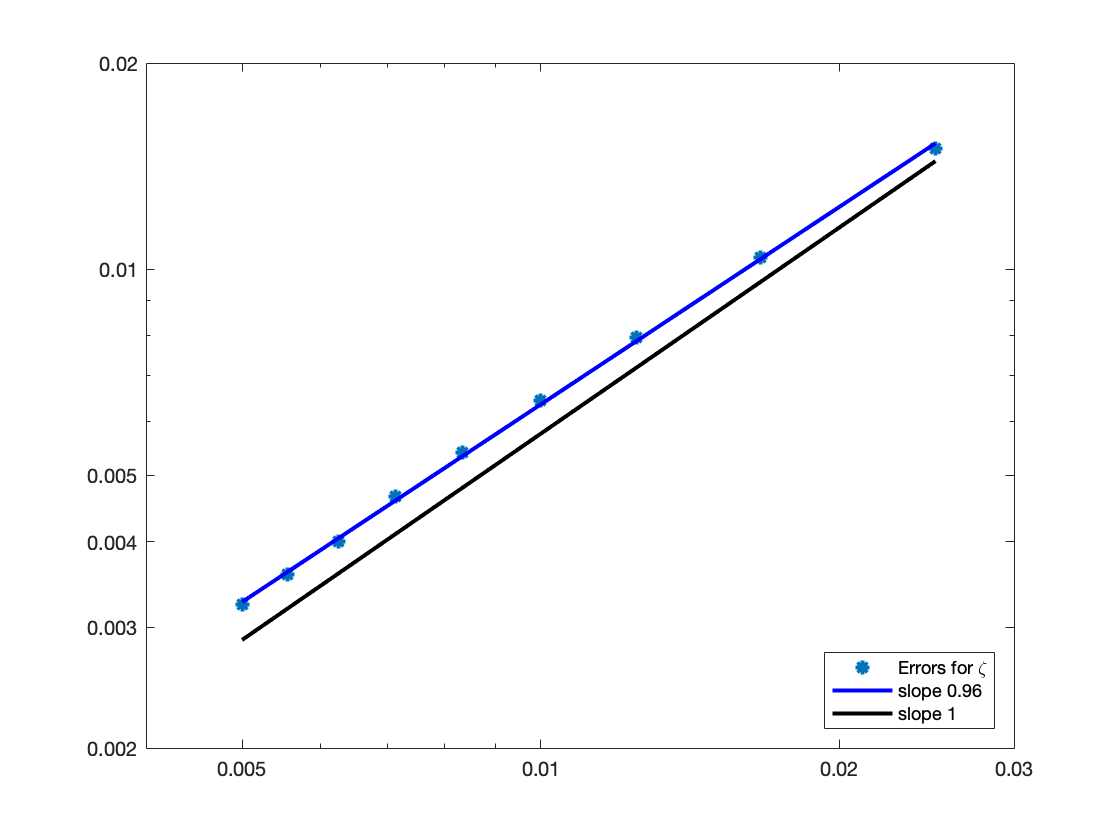}  & \includegraphics[height=50mm]{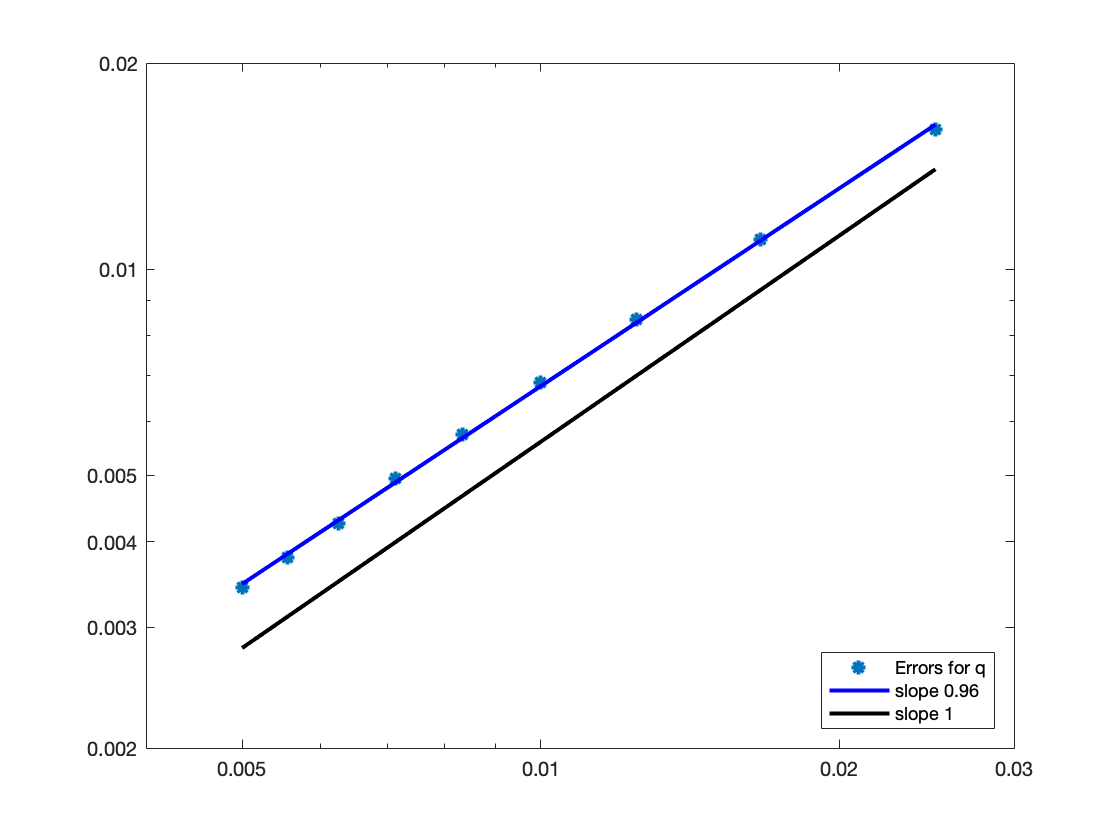}
\end{tabular}
    \caption{Validation of convergence of soliton propagation in the large domain, $L^{\infty}$ error, $\mu = \eps = 0.1$.}
  \label{table_soliton2}
     \end{figure}



To test the imposition of the generating boundary condition we compute the numerical solution of the soliton on the small domain $[0,L]$ with L = 10.   
We use the nonlocal Lax-Friedrichs scheme and a constant time step $\delta_t = 0.8 \, \delta_x$ for $\mu = \eps = 0.3$, and $\delta_t = 0.9 \, \delta_x$ for $\mu = \eps = 0.1$, taking into account the values of  the approximated eigenvalues. 
The space step is computed as $\delta_x = L/n_x$, with $n_x = 400, 600, 800, 1000, 1200, 1400, 1600, 1800, 2000$.
The maximum of the soliton is initially located on the left of the computational domain, at $x = -L/2$, so that the initial datum in the small domain is almost zero, and then the soliton propagates inside it.
The boundary conditions on the left boundary of the small domain are taken into account by imposing the reference solution and its second-order time derivative approximated with the classical centered second-order scheme.
As the initial datum is zero near the right boundary $x = L$, no special effort is necessary for the computation at this boundary if the final time of the simulation is not too large.
The values of $\zeta$ in the small domain at the final time for the reference solution and the numerical solution are  plotted on Figure \ref{res_soliton}.
The numerical results are presented on Figures \ref{table_soliton3} and \ref{table_soliton4}.
On Figure \ref{table_soliton4}, the slope of the linear regression obtained with all error points is completed with the slope of the linear regression obtained with the four more refined error points.
Globally a first-order convergence is observed when the grid is sufficiently refined.

 \begin{figure}[!ht]
\centering
\begin{tabular}{cc}
   \includegraphics[height=50mm]{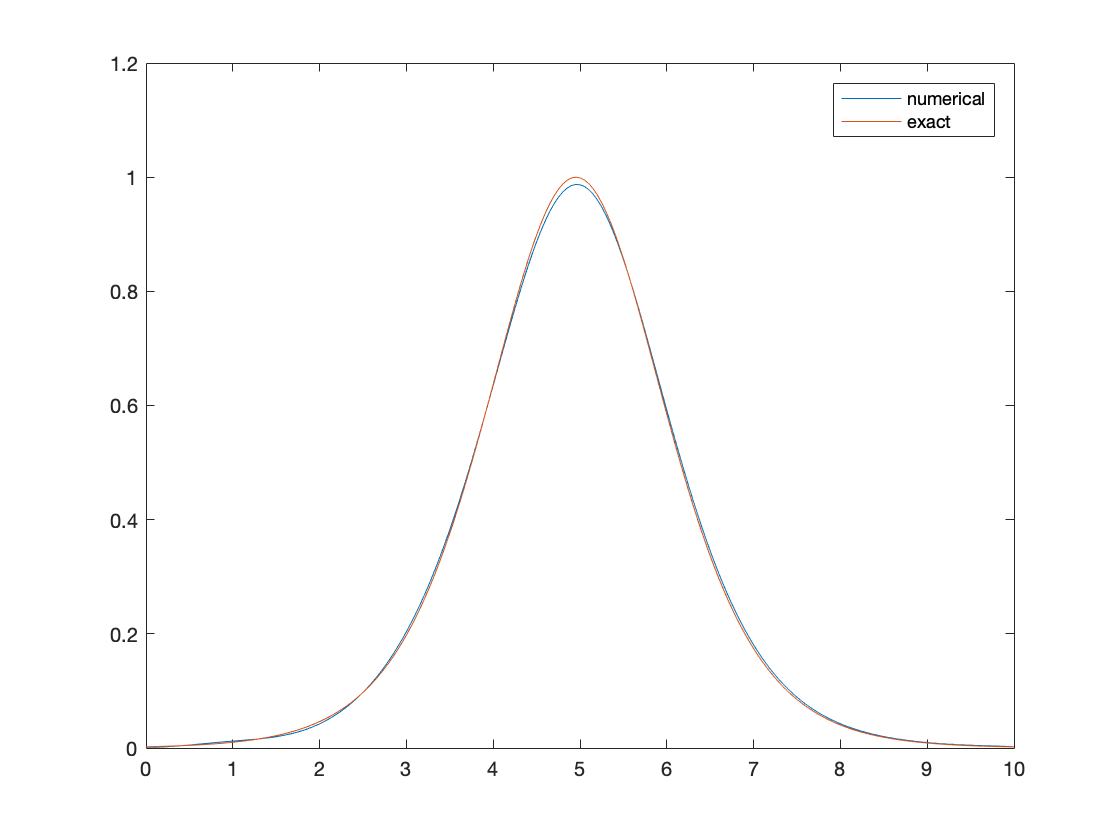}  & \includegraphics[height=50mm]{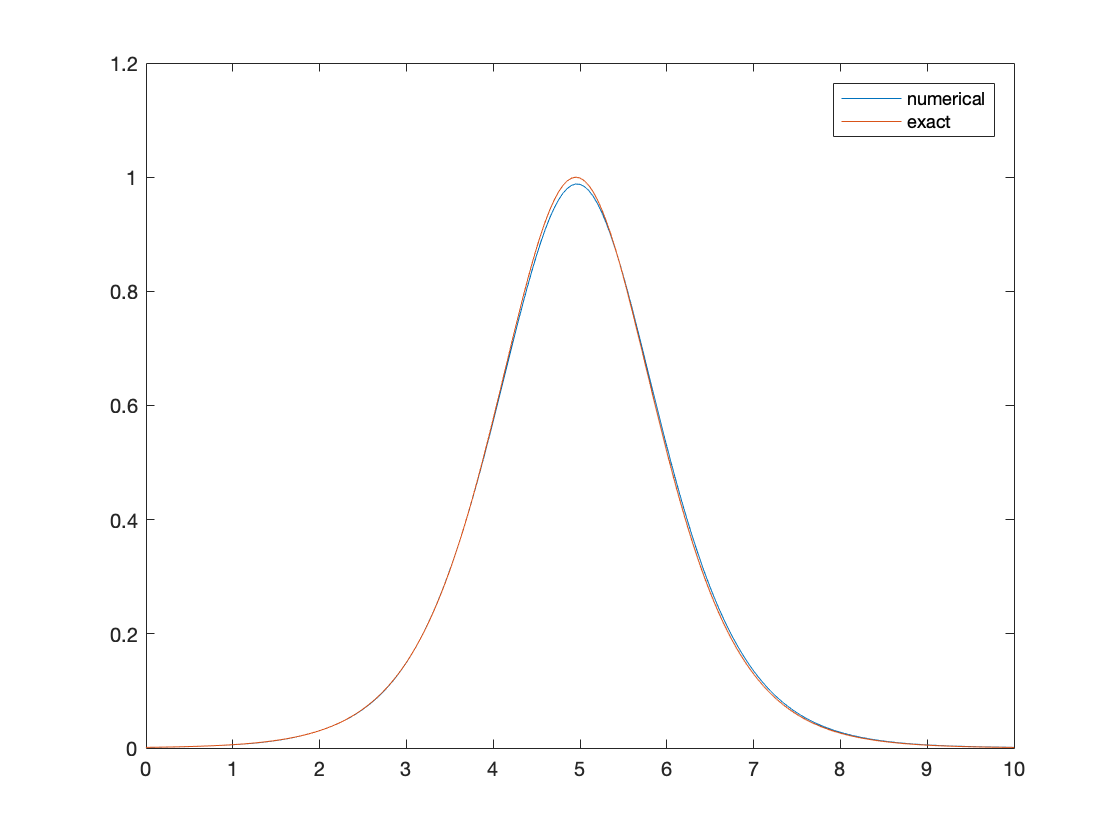}
\end{tabular}
\caption{Soliton: comparison between reference  solution and numerical result for $\zeta$ on the small domain at final time, $\delta_x= L/200$ with $L = 10$, left: $\mu = \eps = 0.3$, right: $\mu = \eps = 0.1$. } { \label{res_soliton}}
   \end{figure}

 \begin{figure}[!ht]
\centering
\begin{tabular}{cc}
   \includegraphics[height=50mm]{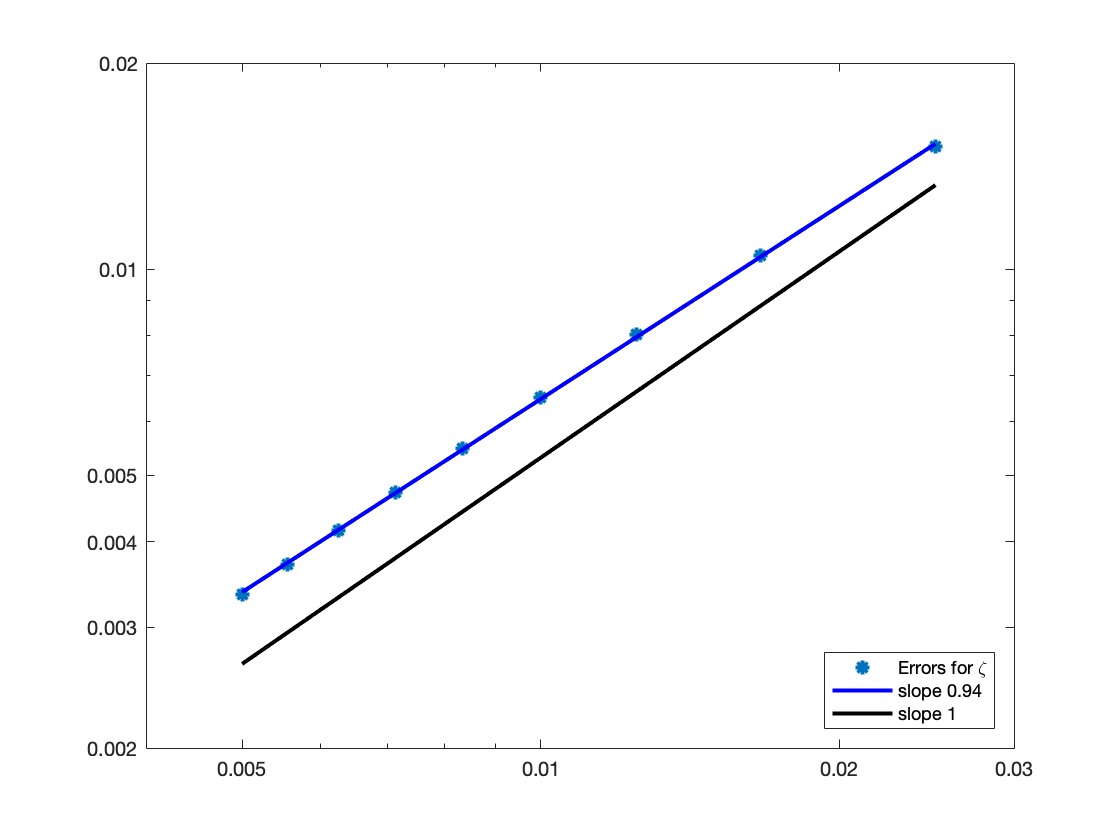}  & \includegraphics[height=50mm]{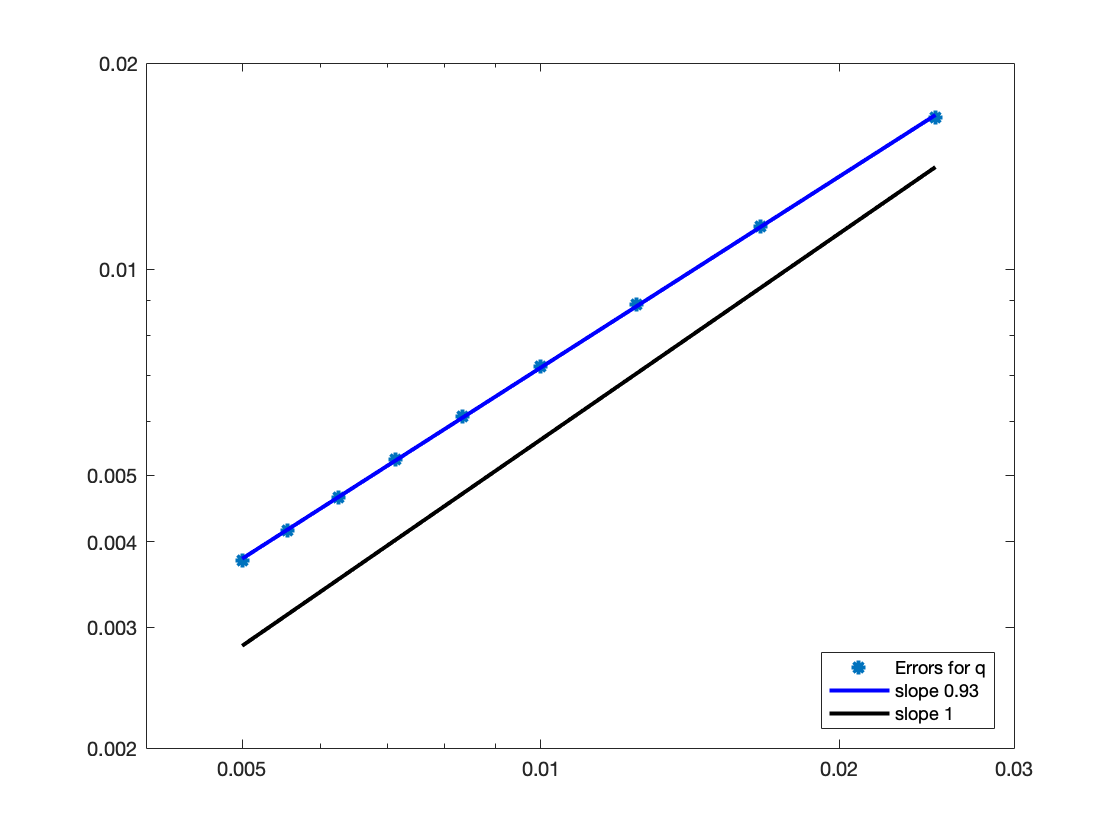}
\end{tabular}
   \caption{Convergence study for the soliton case, $L^{\infty}$ error, $\mu = \eps = 0.3$.}
  \label{table_soliton3}
     \end{figure}

 \begin{figure}[!ht]
\centering
\begin{tabular}{cc}
   \includegraphics[height=50mm]{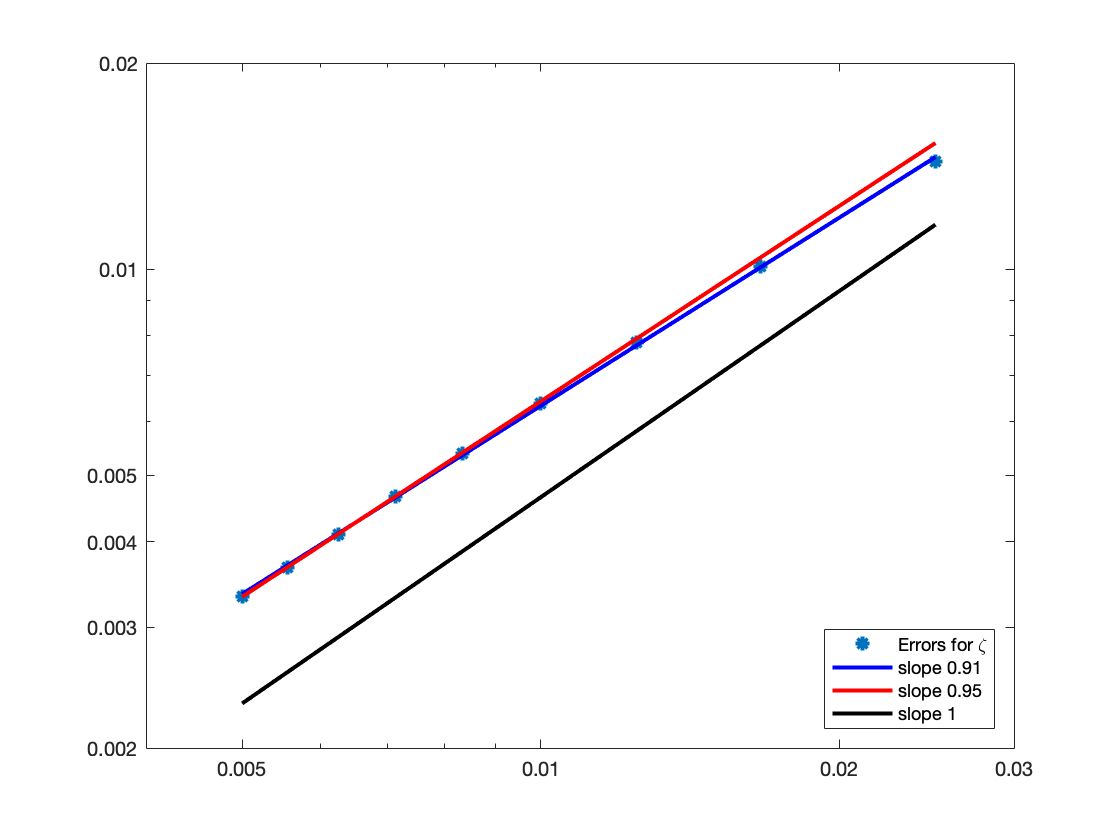}  & \includegraphics[height=50mm]{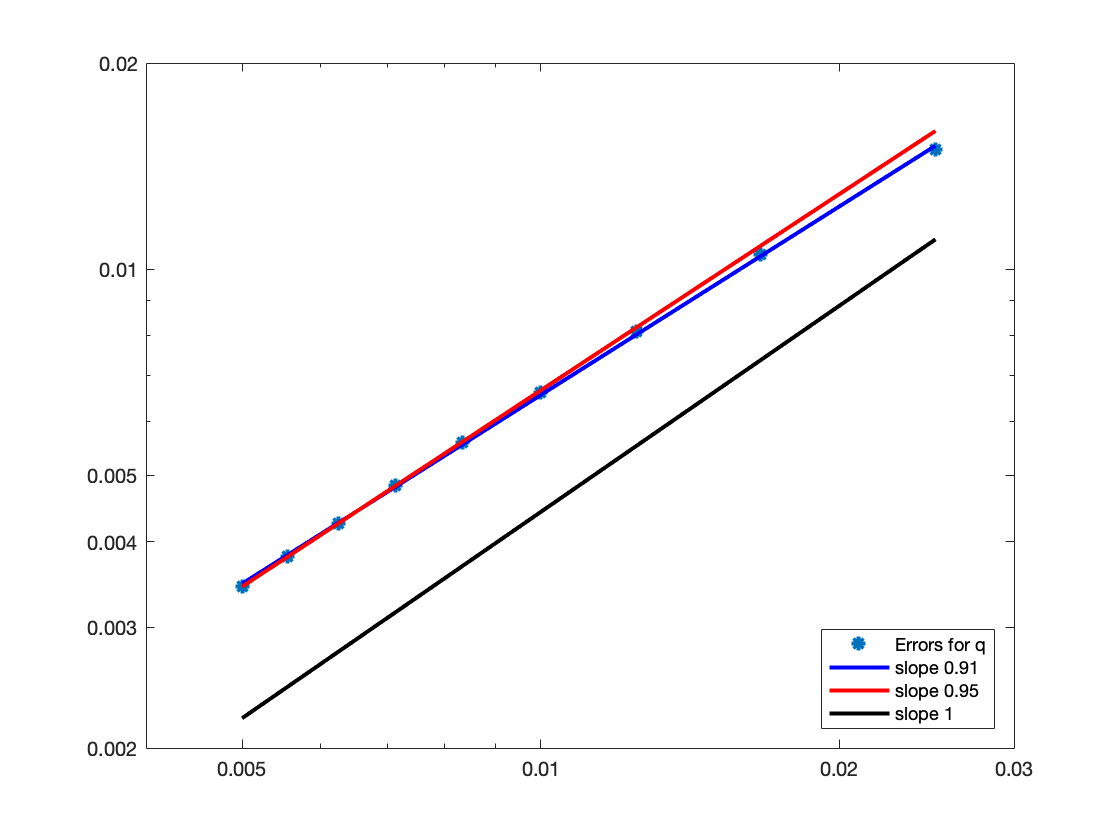}
\end{tabular}
    \caption{Convergence study for the soliton case, $L^{\infty}$ error, $\mu = \eps = 0.1$.}
  \label{table_soliton4}
     \end{figure}

\subsection{Sinusoidal boundary condition }\label{sinus}
We consider the cases
\begin{align*}
({\rm I})\quad  \eps=\mu=0.3,   &\quad  ({\rm II}) \quad \eps=0.1, \quad \mu=0.3 \\
 ({\rm III}) \quad  \eps=\mu=0.1, &\quad  ({\rm IV})\quad  \eps=\mu=0.01.
\end{align*}
Note that case (II) with different values of $\eps$ and $\mu$ has been added here for its relevance for applications in coastal oceanography where a sinusoidal swell is imposed at the entrance of the domain in a region not so shallow (so that $\mu$ is not very small) but the waves are of small amplitude (they become bigger in the shoaling phase, nearer to the shore), so that $\eps$ is small.\\
We first compute a numerical solution $U^L$ with a very refined mesh ($n_x$ =3600) on a larger domain $[-L,L]$, with $L = 10$, with the Lax-Friedrichs scheme and a time step $\delta_t = 0.9 \delta_x$
The initial condition for $U^L$ in the larger domain is
\begin{align}
\zeta^{L}(t=0,x)&= 0; \\
q^L(t=0,x) &= 0,
\end{align}
and we impose until the final time $T_f=15$ the  generating boundary condition 
$$
\zeta(t,x=-L) = \sin (2 \pi t/5).
$$

We define the reference solution on the  slightly smaller domain $[-0.8 \, L,L]$
$$
U^{\rm ref}=(\zeta^{\rm ref},q^{\rm ref})^T:= U^L_{\vert_{[-0.8\, L,L]}}
\quad\mbox{ and }\quad 
f(t):=\zeta^L(t,x= -0.8\, L).
$$
Then we compute a solution with coarse meshes on this smaller domain $[-0.8 \, L,L]$.
The boundary conditions at $x=-0.8 \, L$ are taken into account by imposing the reference solution and its second-order time derivative approximated with the classical centered second-order scheme.
No special effort is necessary for the computation with the coarse mesh at the right boundary $x=L$ until the final time $T_f = 15$.
The values of $\zeta$ in the small domain at the final time for the reference solution and the numerical solution are  plotted on Figure \ref{res_sinus}.
Because there is numerical dissipation, the numerical solution has a smaller amplitude than the reference solution after some time of propagation inside the small domain, but both solutions coincide well near the left boundary.
The error between the reference solution and the solution on the coarse mesh is computed near the left boundary on the interval $ [-0.8\, L , -0.6\, L]$ in order to measure the error due to the generating boundary condition rather than the dissipation error inherent to the Lax-Friedrichs scheme (which stronger in this numerical test than in the previous ones due to the fact the the reference solution involves higher frequencies).  

The numerical results for the cases $({\rm I})$ and $({\rm II})$  are presented on Figures \ref{table_sinus1} and \ref{table_sinus1bis}.
The space steps $\delta_x$ were chosen as $\delta_x = 2L/n_x$, with $n_x$ = 100, 120, 150, 200, 240, 300, 360, 400, 600, 720, 900.
On both Figures the slope of the linear regression obtained with all error points is completed with the slope of the linear regression obtained with the three more refined error points.
The numerical results for the cases $({\rm III})$ and $({\rm IV})$ are presented on \ref{table_sinus2} and \ref{table_sinus3}.
The space steps $\delta_x$ were chosen as $\delta_x = 2L/n_x$, with $n_x = $ 90, 100, 120, 150, 200, 240, 300, 360, 400.
 Globally, a first-order convergence is observed when the grid is sufficiently refined.
%

 \begin{figure}[!ht]
\centering
\begin{tabular}{cc}
   \includegraphics[height=50mm]{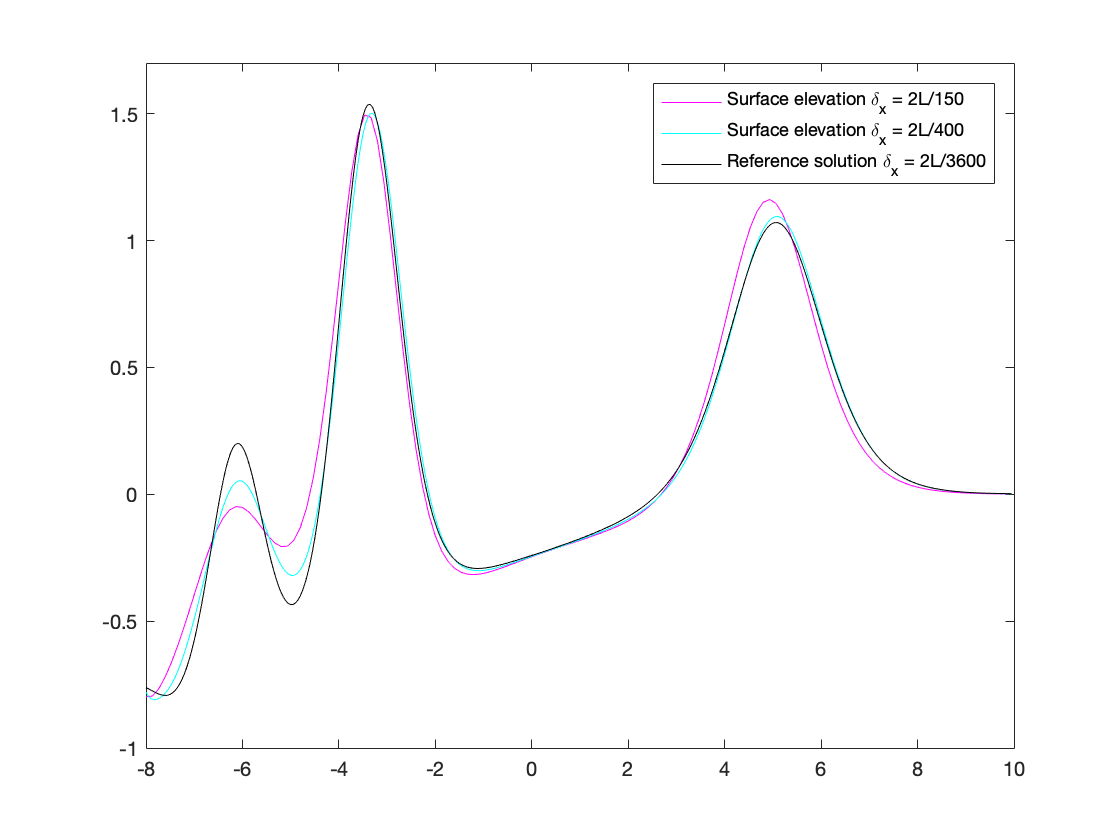}  & \includegraphics[height=50mm]{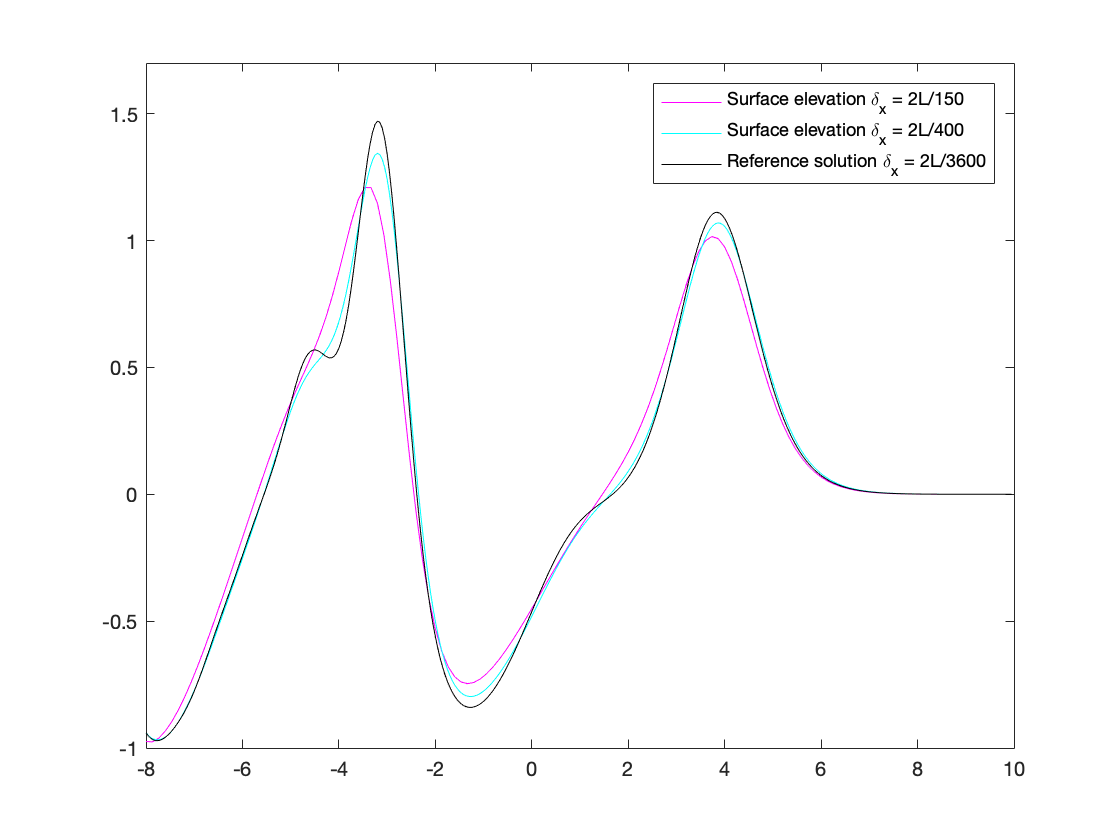}
\end{tabular}
\caption{Sinusoidal boundary condition: comparison between reference  solution and numerical result for $\zeta$ on the small domain at final time $T_f=15$, $\delta_x= 2L/150, 2L/400$ with $L = 10$, left: $\mu = \eps = 0.3$, right: $\mu = \eps = 0.1$. (The numerical solution with $\delta_x= 2L/3600$ coincide with the reference solution) }  \label{res_sinus}
   \end{figure}

 \begin{figure}[!ht]
\centering
\begin{tabular}{cc}
   \includegraphics[height=50mm]{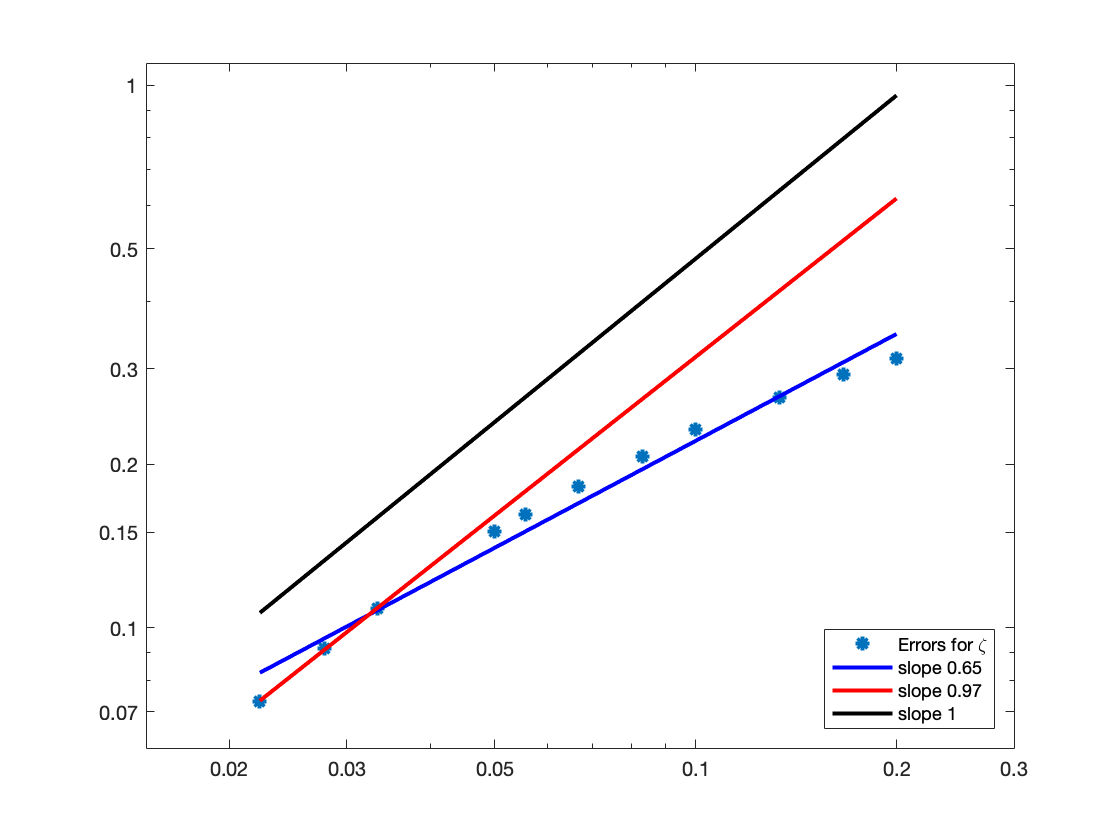}  & \includegraphics[height=50mm]{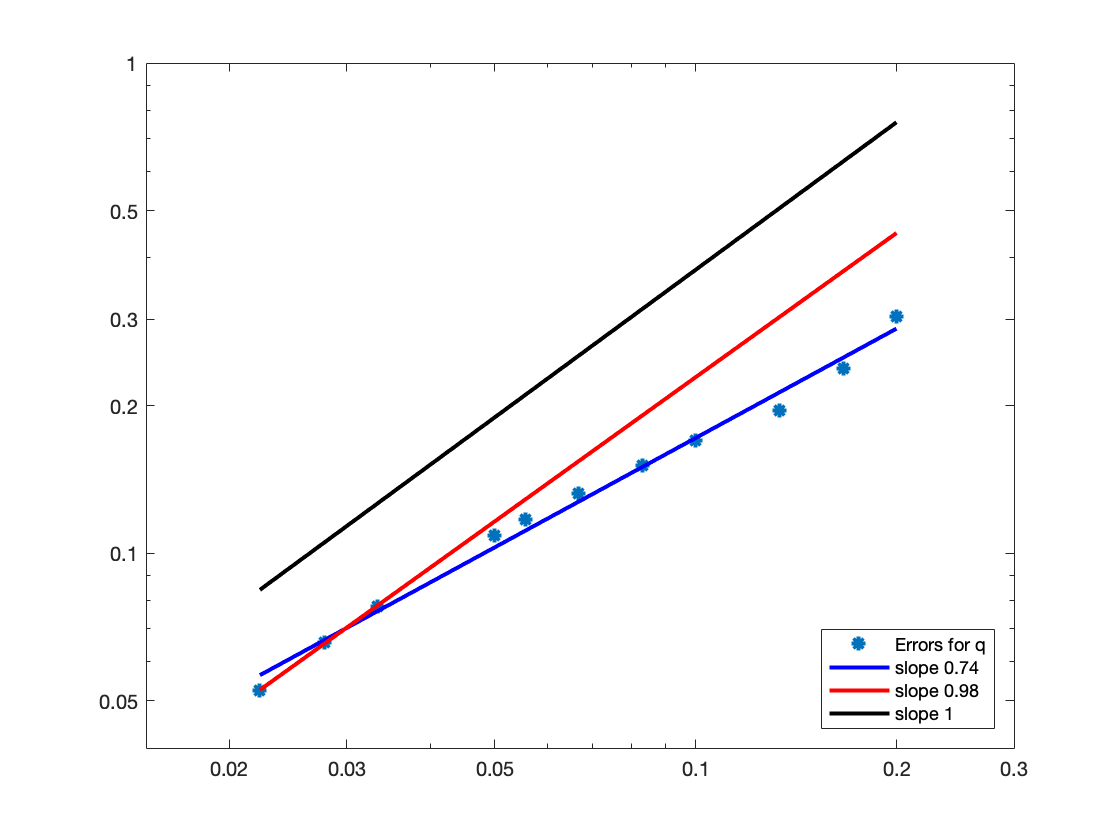}
\end{tabular}
 \caption{Convergence study for the sinusoidal boundary condition, $L^{\infty}$ error, $\mu = \eps = 0.3$.}
  \label{table_sinus1}
     \end{figure}

 \begin{figure}[!ht]
\centering
\begin{tabular}{cc}
   \includegraphics[height=50mm]{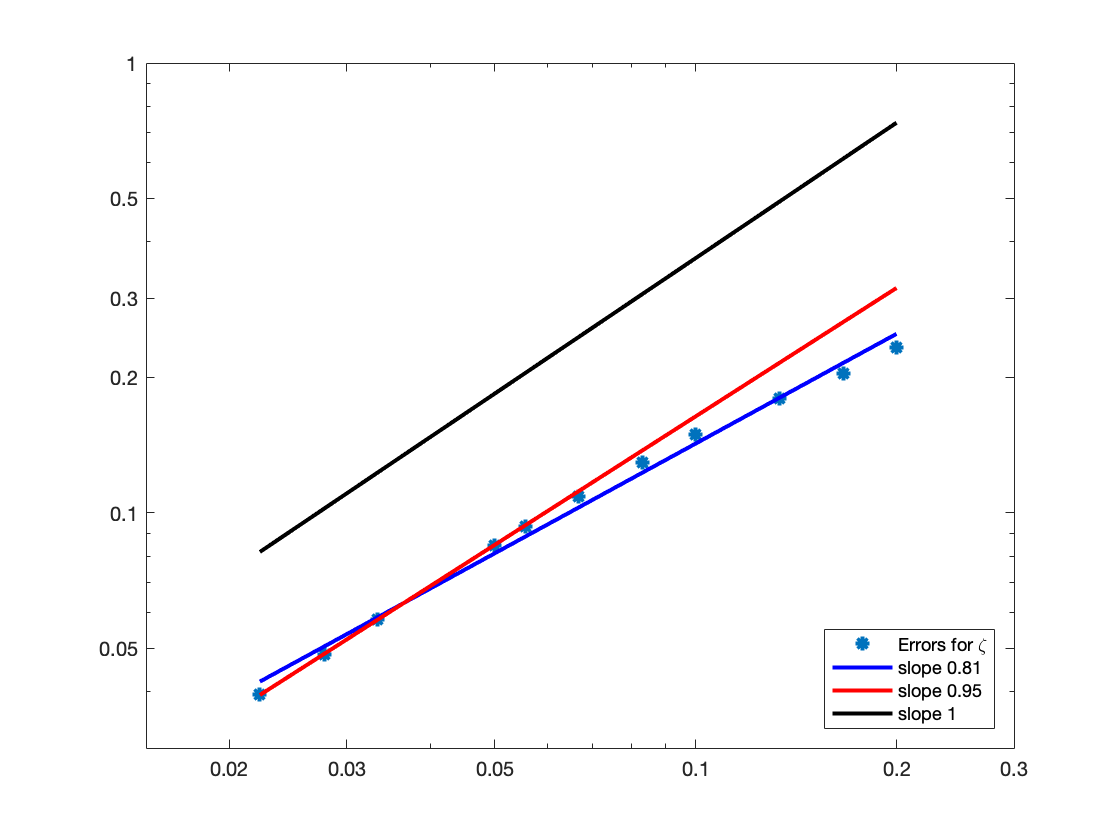}  & \includegraphics[height=50mm]{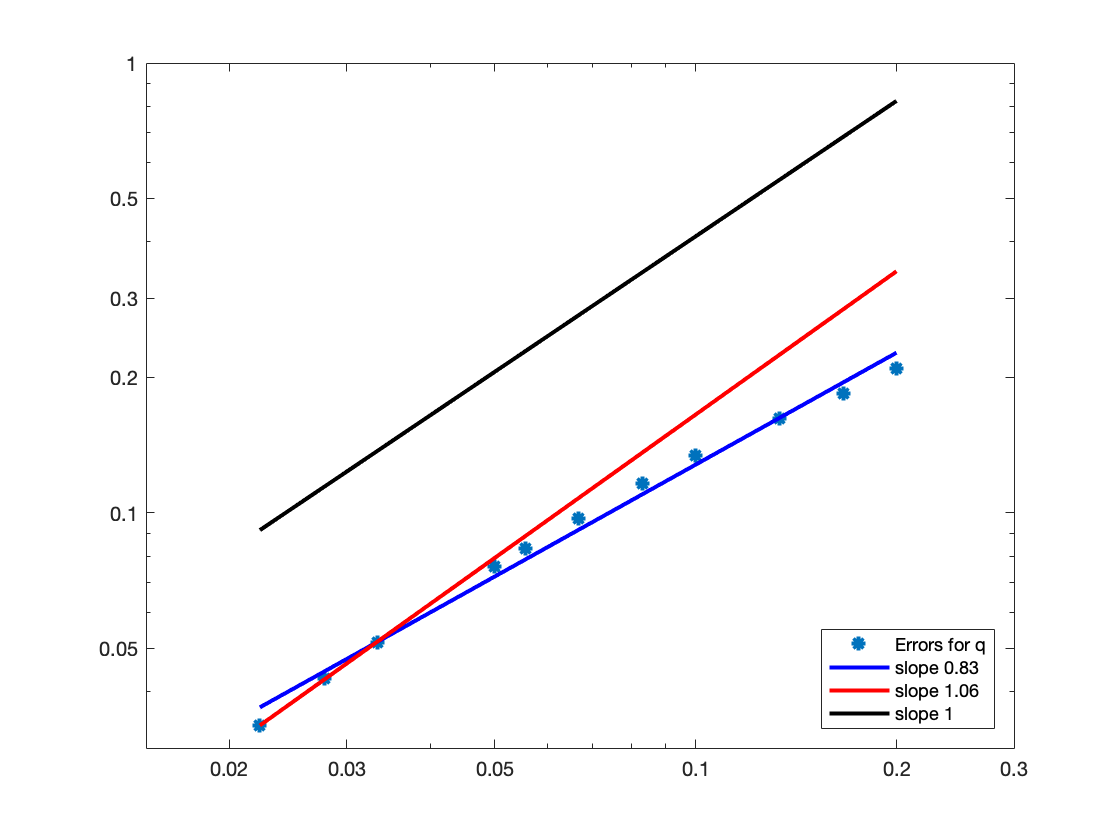}
\end{tabular}
 \caption{Convergence study for the sinusoidal boundary condition, $L^{\infty}$ error, $\mu = 0.3, \eps = 0.1$.}
  \label{table_sinus1bis}
     \end{figure}

 \begin{figure}[!ht]
\centering
\begin{tabular}{cc}
   \includegraphics[height=50mm]{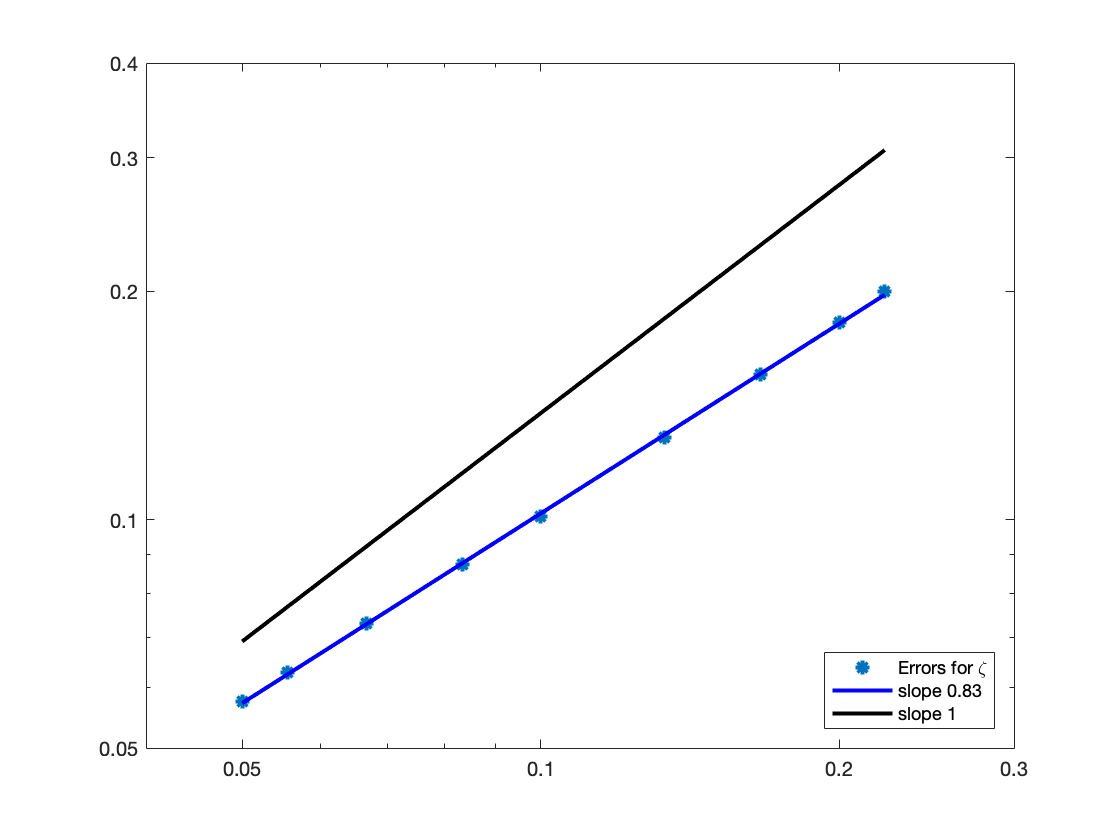}  & \includegraphics[height=50mm]{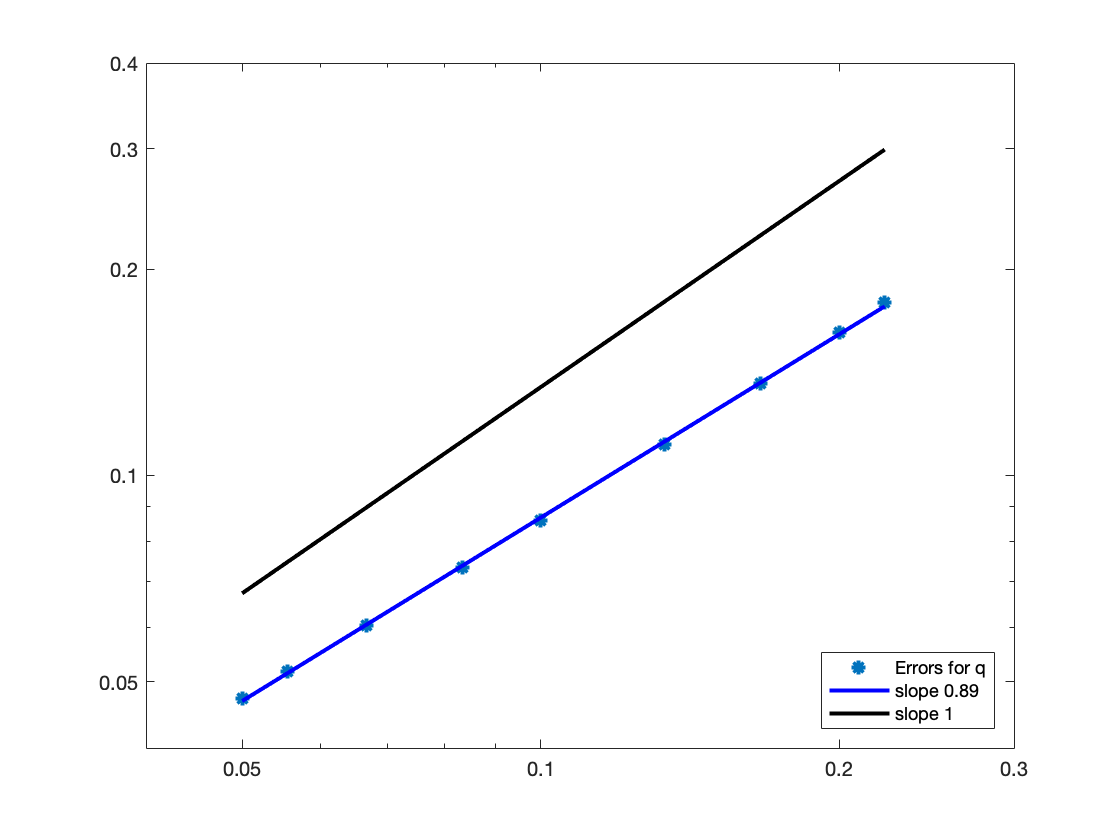}
\end{tabular}
 \caption{Convergence study for the sinusoidal boundary condition, $L^{\infty}$ error, $\mu = \eps = 0.1$.}
  \label{table_sinus2}
     \end{figure}

\begin{figure}[!ht]
\centering
\begin{tabular}{cc}
   \includegraphics[height=50mm]{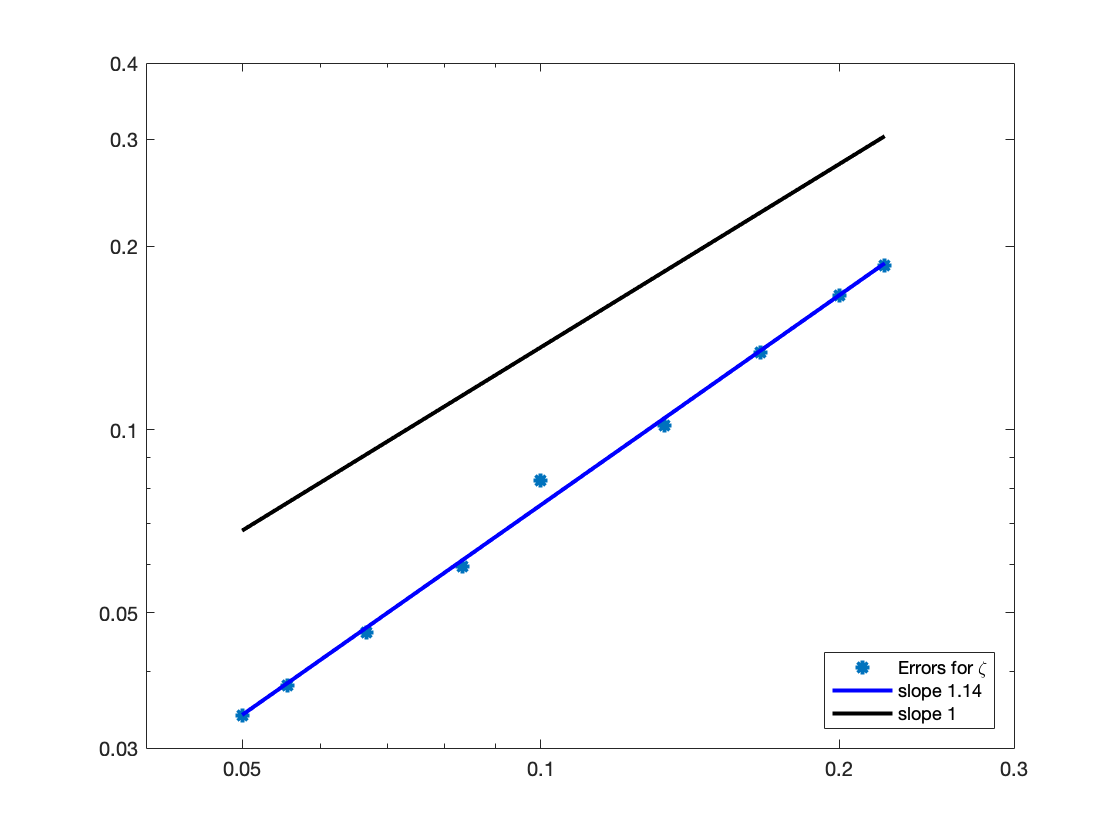}  & \includegraphics[height=50mm]{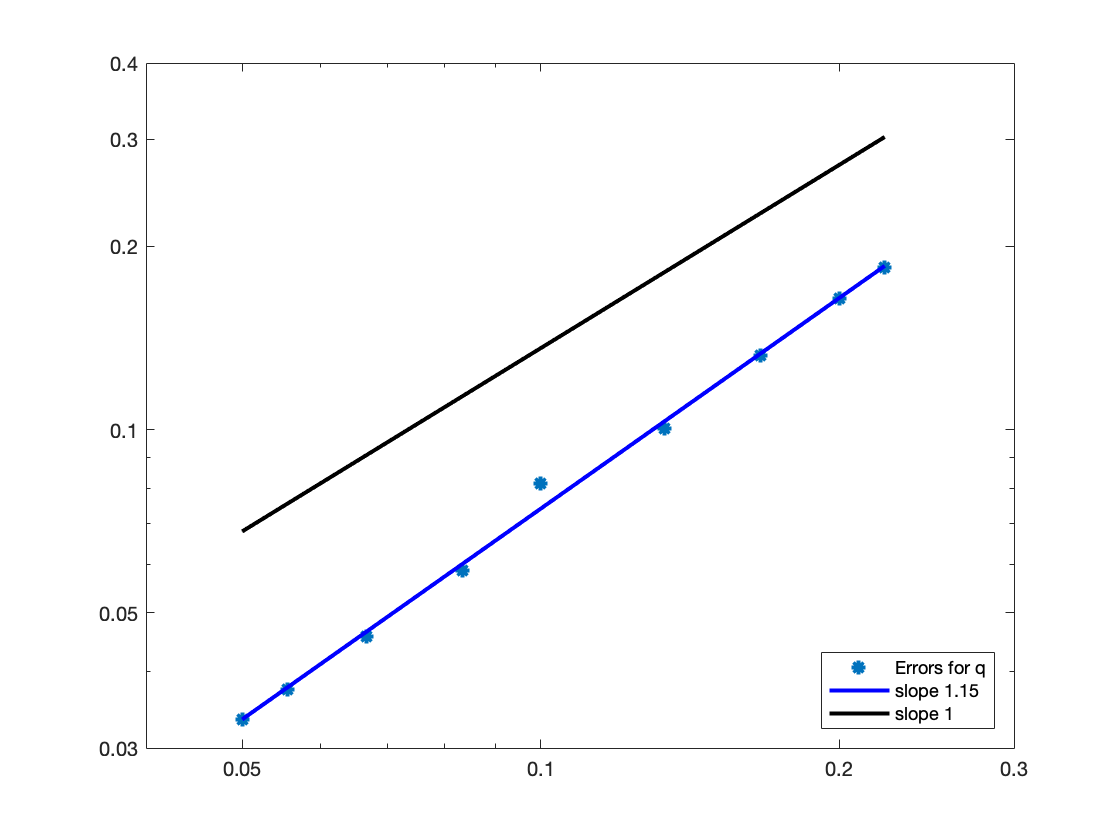}
\end{tabular}
 \caption{Convergence study for the sinusoidal boundary condition, $L^{\infty}$ error at final time, $\mu = \eps = 0.01$.}
  \label{table_sinus3}
     \end{figure}

\end{document}